\documentclass[12pt]{amsart}
 \usepackage{amsfonts,amssymb,eucal}
\usepackage{bbold,mathbbol,bbm}
\usepackage{amsthm} 
 \usepackage{amsmath} 
\usepackage{amscd}
 \usepackage{latexsym} 
\input{cyracc.def} 
\font\tencyr=wncyr10
\def\cyr{\tencyr\cyracc}
\numberwithin{equation}{section}
\usepackage{color}
\newcommand{\dd}{\operatorname{d}}

\newcommand{\e}{{\mbox{\rm e}}}

\newcommand{\mb}[1]{{\mbox{\boldmath{$#1$}}}}
\newcommand{\mc}[1]{{\mathcal{#1}}}
\newcommand{\got}[1]{{\mathfrak{#1}}}
\newcommand{\db}[1]{{\mathbb{#1}}}

\newcommand{\pa}{\partial}
\newcommand{\R}{\ensuremath{\mathbb{R}}}
\newcommand{\C}{\ensuremath{\mathbb{C}}}
\newcommand{\N}{\ensuremath{\mathbb{N}}}

\newcommand{\Hi}{\ensuremath{\got{H}}}

\newcommand{\Z}{\ensuremath{\mathbb{Z}}}
\newcommand{\Hinf}{\ensuremath{\mathcal{H}^{\infty}}}

\renewcommand{\P}{\ensuremath{\mathbb{P}}}
\newcommand{\Phinf}{\ensuremath{\P (\Hinf )}}
\newtheorem{Remark}{Remark}

\newcommand{\fl}{\ensuremath{{\got{F}}_{\Hi}}}
\newcommand{\un}{\ensuremath{\mathbb{1}_n}}
\newtheorem{Proposition}{Proposition}

\theoremstyle{definition}

\newcommand{\gl}{\ensuremath{\mc{L}}}
\newcommand{\gf}{\ensuremath{\mc{F}}}
\newcommand{\gcp}{\ensuremath{\db{CP}}}
\newcommand{\gpl}{\mbox{{\bf P}(\gl )}}

\begin{document}
\title{Bergman representative coordinates on the Siegel-Jacobi disk} 
\author{Stefan  Berceanu}
\address[Stefan  Berceanu]{National
 Institute for Physics and Nuclear Engineering\\
         Department of Theoretical Physics\\
         PO BOX MG-6, Bucharest-Magurele, Romania}
\email{Berceanu@theory.nipne.ro}

\begin{abstract}
We underline some differences between the geometric aspect  of
Berezin's approach  to quantization on homogeneous K\"ahler manifolds 
and Bergman's construction for bounded domains in $\mathbb{C}^n$.
We construct explicitly the Bergman representative coordinates for the
Siegel-Jacobi disk $\mathcal{D}^J_1$, which is a partially bounded manifold whose
points belong to  $\mathbb{C}\times\mathcal{D}_1$, where $\mathcal{D}_1$
denotes the Siegel disk. The Bergman representative coordinates on
$\mathcal{D}^J_1$  are globally defined, the Siegel-Jacobi disk is  a
normal  K\"ahler homogeneous  Lu Qi-Keng manifold,  whose 
representative manifold is the Siegel-Jacobi disk itself.
\end{abstract}
\maketitle
\noindent
\tableofcontents
\newpage
\section{Introduction}\label{intro}

In this paper we discuss several geometric issues   of   Berezin's approach to quantization on
homogeneous K\"ahler manifolds. The starting
point is a weighted Hilbert space of square integrable holomorphic functions $\Hi_f$
defined on a homogeneous K\"ahler manifold $M$ of weight $\e^{-f}$,  where the
K\"ahler potential $f$ is  the logarithm of the kernel function
$K_M$, which is obtained as the scalar product of two coherent states
(CS) defined on
$M$ \cite{perG}. The metric $\dd s^2_M$ associated with $f=\ln K_M$ is
in fact the so called 
{\it balanced metric} \cite{don,arr,alo}.
The  metric $\dd s^2_M$ is different from the  Bergman metric $\dd s^2_{\mc{B}_2}$, 
defined  on bounded domains \cite{berg,berg1,berg2} or
in  Kobayashi's extension to manifolds, $\dd s^2_{\mc{B}_n}$,  based on a Hilbert space of square integrable
top degree holomorphic forms $\gf_n(M)$ defined on $M$ \cite{koba}. 

In a series of papers  starting with \cite{jac1,sbj,nou} we have constructed CS
\cite{perG} attached to the Jacobi group $G^J_n=
H_n\rtimes\text{Sp}(n,\R)$, where $H_n$ denotes the 
$(2n+1)$-dimensional Heisenberg group \cite{ez,bs}. The Jacobi group
is important in Quantum Mechanics, being responsible  for the squeezed states in
Quantum Optics, see references in \cite{jac1,sbj,sbcg,ber9,gem,FC1}. The homogeneous K\"ahler manifolds
$\mc{D}^J_n$  are attached to
the Jacobi group $G^J_n$ \cite{satake,gem,Y10}.
The  Siegal-Jacobi
ball $\mc{D}^J_n$  is not a bounded domain,  but 
     {\it partially bounded}  (this denomination is  borrowed from
     \cite{Y08,Y10}), 
its points being in $\C^n\times \mc{D}_n$, where $\mc{D}_n=\text{Sp}(n,\R)_{\C}/\text{U}(n)$ denotes the
Siegel bounded domain (Siegel ball) \cite{sbj,nou,Y08,Y10}.

We introduce   the term {\it normal manifolds}  of
Lichnerowicz \cite{Lich}  to designate manifolds for which $K_M(z)\not= 0,
\forall  z\in M$.  Also we use 
an  extension of the meaning of {\it Kobayashi manifolds}  which
appears in the book  \cite{ps1} of
Piatetski-Shapiro, and  we advance   the  denomination  {\it Lu Qi-Keng
  manifolds} to designate  manifolds on which  $K_M(z,\bar{w})\not= 0, \forall
z,w\in M$, extending to manifolds  the notion  introduced by  Lu
Qi-Keng  for domains \cite{lu66}.  

In this paper we 
construct explicitly the Bergman representative coordinates for the
Siegel-Jacobi disk  $\mc{D}^J_1$. Usually,  the Bergman representative coordinates
are defined on bounded domains $\mc{D}\subset\C^n$ \cite{berg}.  We show
that the Bergman representative coordinates on $\mc{D}^J_1$ are globally
defined and  the Siegel-Jacobi disk is  a homogeneous K\"ahler  Lu
Qi-Keng  manifold,  whose
representative manifold is  the Siegel-Jacobi disk itself.

The paper is laid out as follows. Berezin's recipe of
quantization \cite{ber73,ber74,ber75} using CS 
\cite{perG}  adopted in \cite{sb6,jac1,sbj,nou,csg}
is presented  in \S \ref{gi} in the  geometric
setting of the papers \cite{don,alo,arr}, where the 
$\epsilon $-function \cite{raw,Cah,cah} is constant. We underline the difference between the
balanced metric and the Bergman metric \cite{berg,berg1}. \S
\ref{TSJD} summarizes the geometric information on the Siegel-Jacobi
disk extracted from
\cite{ jac1,nou,FC,csg} in a more precise formulation. In Proposition
\ref{prop2} we have included   the expression of 
the Laplace-Beltrami operator on the Siegel-Jacobi disk, also
obtained  in \cite{Y10}. In \S \ref{EMSS} we discuss the embedding of
the Siegel-Jacobi disk in an infinite-dimensional projective Hilbert
space, underling that the Jacobi group $G^J_1$ is a CS-group
\cite{lis,neeb}. We emphasize  the difference of this embedding
comparatively with 
Kobayashi embedding \cite{koba}. \S \ref{BRC} is devoted to the  Bergman
representative coordinates \cite{berg}. Firstly the definition of the
Bergman representative
coordinates for homogeneous K\"ahler manifolds is given.Then  some general properties of the
Bergman representative coordinates are
recalled in Remark \ref{RRE}. In \S \ref{ssimp} the simplest example of
the Siegel disk is presented. The notion of Lu Qi-Keng manifold is
introduced.
 In  \cite{ber97}  we have called  the set
$\Sigma_z:=\{w\in M|K_M(z,\bar{w})=0\}$  {\it{polar divisor of}}
$z\in M$, underling its meaning in algebraic geometry \cite{gh}, and its 
equality with the cut locus \cite{koba2} for some compact homogeneous manifolds. 
 In Proposition \ref{PR1} it is underlined that the
Siegel-Jacobi disk is a Lu Qi-Keng manifold and the Bergman representative
coordinates are globally defined on it.  In  Proposition 
\ref{REMFF} it is proved that the representative manifold of the
Siegel-Jacobi disk is the Siegel-Jacobi disk itself. The Appendix \S  \ref{bsmb} is dedicated to the Bergman
pseudometric and metric. Also the notion of Kobayashi
embedding \cite{koba} and Kobayashi manifold \cite{ps1} are recalled
in   \S \ref{EMBB}.

{\bf Notation}. 
$\R$, $\C$, $\Z$  and $\N$ denotes the fields of real,
complex numbers, the set of non-negative integers,  respectively the ring of  integers.  We denote the imaginary unit
$\sqrt{-1}$ by $i$, and the Real and Imaginary part of a complex
number by $\Re$ and respectively $\Im$, i.e. we have for $z\in\C$,
$z=\Re z+i \Im z$, and $\bar{z}=\Re z-i \Im z$, also denoted
$cc(z)$.
 In this paper the Hilbert space $\Hi$ is endowed with
a scalar product $(\cdot,\cdot)$ 
antilinear in the first argument, i.e. $(\lambda x,y)=\bar{\lambda}(x,y)$,
$x,y\in\Hi,\lambda\in\C\setminus 0$.
 We denote by $\dd$ the differential, and we have $\dd =\pa
 +\bar{\pa}$, where, in a local system of coordinates
 $(z_1,\dots,z_n)$ on a complex manifold $M$, we have $\pa f=
 \sum_{i=1}^n\frac{\pa f}{\pa z_i}\dd z_i$. 

\section{Berezin's quantization: the starting point}\label{gi}
Firstly we    highlight  the relationship of the  approach
to geometry  adopted 
in the papers \cite{jac1,sbj,nou,FC,csg} in the context of CS 
defined in mathematical physics \cite{ber74,perG} with the ``pure''  geometric
approach of mathematicians \cite{berg,berg1,berg2,berg3,koba,Lich,arr}.

Let $M$ be a  complex manifold  of complex dimension $n$. $M$ is
called {\it{hermitian}} if a hermitian structure $\operatorname{H}$  is given in its
tangent bundle $T(M)$ \cite{chern}. If we choose  a local coordinate system
$(z_1,\dots,z_n)$, then a natural frame is given by $\frac{\pa}{\pa z_i}$,
$i=1,\dots,n$. Let us denote $h_{i\bar{j}}:=\operatorname{H}(\frac{\pa}{\pa z_i},\frac{\pa}{\pa
  \bar{z}_j}), i,j=1,\dots, n $, 
and the matrix $h$ is positive definite hermitian.  A hermitian
manifold is called {{\it K\"ahlerian}}  if its  K\"ahler two-form, i.e. the   real-valued  $(1,1)$-form
\begin{equation}\label{kall}
\omega_M(z)=i\sum_{i,j=1}^{n} h_{i\bar{j}} (z) \dd z_{i}\wedge
\dd\bar{z}_{j}, 
\end{equation}
is closed,  i.e. $\dd \omega_M=0$.  Equivalently (see e.g. Theorem C at
p. 56 in \cite{chern}),  {\it{a hermitian manifold is K\"ahlerian if
and only if there exists a real-valued $C^{\infty}$ function $f$ - the
 K\"ahler potential -  such
that}} 
\begin{equation}\label{OMEGAMM}
\omega_M=i \pa\bar{\pa} f, \end{equation}
{\it{ and   in \eqref{kall} we have}}
\begin{equation}\label{Ter}
h_{i\bar{j}}(z)= \frac{\pa^2 f}{\pa {z}_{i}\pa \bar{z}_j}.
\end{equation}

An {\it{automorphism}} of the K\"ahler manifold $M$ is an
invertible holomorphic mapping preserving the Hermitian structure $\operatorname{H}$. The
K\"ahler manifold $M$ is called {\it{homogeneous}} if the group $G(M)$
of all automorphisms of  $M$ acts transitively on it, see e.g. \cite{gpv}. A  ``domain''
is a connected open subset of $\C^n$, $\mc{D}\subset\C^n$,  and ``bounded'' means 
relatively compact. The Bergman metric (see \cite{berg1,FUKs2,AWeil} and
\eqref{berg11} in 
the Appendix) defines in $\mc{D}$ a canonical
K\"ahlerian structure,  and $G_A(\mc{D})= G(\mc{D})$, where
  $G_A(M)$ denotes the group of invertible holomorphic transformation
  of the manifold $M$ \cite{gpv}.

Let us consider the
weighted Hilbert space $\Hi_f$ of square integrable holomorphic functions on
$M$, with weight $\e^{-f}$ 
\begin{equation}\label{HIF}
\Hi_f=\left\{\phi\in\text{hol}(M) | \int_M\e^{-f}|\phi|^2 \Omega_M
  <\infty \right\},
\end{equation}
where $\Omega_M$ is the volume form
\begin{equation}\label{OMM}
\Omega_M:=\frac {1}{n!}\;
\underbrace{\omega\wedge\ldots\wedge\omega}_{\text{$n$ times}},
\end{equation}
$G$-invariant for $G$-homogeneous manifolds $M$.

If $\Hi_f\not= 0$, let  $\left\{\varphi_j(z)\right\}_{j=0,1,\dots}$ be  an orthonormal  base of functions
of $\Hi_f$ and define the  kernel function by 
\begin{equation}\label{sumBERG}
K_M(z,\bar{z})=\sum_{i=0}^{\infty}\varphi_i(z)\bar{\varphi}_i(z).
\end{equation}
For compact manifolds $M$, $\Hi_f$ is finite dimensional  and 
the sum \eqref{sumBERG} is finite.

In order to fix the  notation on CS \cite{perG}, let us consider the triplet $(G, \pi , \got{H} )$, where $\pi$ is
 a continuous, unitary, irreducible 
representation
 of the  Lie group $G$
 on the   separable  complex  Hilbert space \Hi . 

Let us now denote by $H$  the isotropy group with Lie subalgebra
$\got{h}$ of the Lie algebra $\got{g}$ of $G$.  
We  consider (generalized)  CS  on complex  homogeneous manifolds $M\cong
G/H$ \cite{perG}.  
{\em The coherent vector
 mapping}  $\varphi$ is defined locally, on a coordinate neighborhood
$\mc{V}_0\subset M$ (cf. \cite{last,sb6}):
 \begin{equation}\label{ECFI}\varphi : M\rightarrow \bar{\Hi}, ~
 \varphi(z)=e_{\bar{z}},
\end{equation}
where $ \bar{\Hi}$ denotes the Hilbert space conjugate to $\Hi$.
The  vectors $e_{\bar{z}}\in\bar{\Hi}$ indexed by the points
 $z \in M $ are called  {\it
Perelomov's  CS vectors}. Using Perelomov's CS
vectors, we consider Berezin's approach to quantization on K\"ahler
manifolds with the supercomplete set of vectors
verifying the Parceval   overcompletness identity \cite{ber73}-\cite{berezin},
\begin{equation}\label{PAR}
(\psi_1,\psi_2)_{\Hi^*}=\int_M (\psi_1,e_{\bar{z}}) (e_{\bar{z}},\psi_2) \dd
\nu_M(z,\bar{z}), \quad \psi_1,\psi_2\in\Hi ,
\end{equation}
where we have identified the space $\overline{\Hi}$  complex conjugate
 to \Hi~  with the dual
space
$\Hi^{\star}$ of $\Hi$. 
$\dd{\nu}_{M}$  in \eqref{PAR} is the quasi-invariant measure on $M$
given by 
\begin{equation}\label{DELNU}
\dd{\nu}_{M}(z,\bar{z})=\frac{\Omega_M(z,\bar{z})}{(e_{\bar{z}},e_{\bar{z}})},
\end{equation}
where 
 $\Omega_M$ is the normalized  $G$-invariant volume form
 \eqref{OMM}. In fact, \eqref{PAR}, \eqref{DELNU}  and
 \eqref{OMM} are formula  (2.1)  and the next one on p. 1125 in
 Berezin's paper \cite{ber74}. 

If the $G$-invariant  K\"ahler two-form $\omega$ on the (real) $2n$-dimensional
manifold  $M=G/H$  has the expression  \eqref{kall}, 
then in \eqref{OMM}  (see
e.g.  (4.2) in  \cite{ball}) we have:
\begin{equation}\label{nomega}
\underbrace{\omega\wedge\ldots\wedge\omega}_{\text{$n$ times}}= i^n
n !~ \mc{G}(z) \dd z_1\wedge\dd \bar{z}_1\wedge\dots\wedge \dd z_n\wedge\dd\bar{z}_n, 
\end{equation}
 where the density of the $G$-invariant volume $\mc{G}(z)$ has the
expression: 
\begin{equation}\label{DETG}
~ \mc{G}(z):=\det
(h_{i\bar{j}}(z))_{i,j=1,\dots,n}.
\end{equation}
If $z_j=x_j+i y_j$, $j=1,\dots,n$, then we have the relations (see
e.g. (5) on p. 14 in \cite{AWeil}):
\begin{equation}\label{VPOS}
2^nx_1\wedge y_1\wedge \dots\wedge x_n\wedge
y_n=i^nz_1\wedge\bar{z_1}\wedge\dots\wedge
z_n\wedge\bar{z_n}=i^{n^2}
z_1\wedge\dots\wedge z_n\wedge\bar{ z}_1\wedge\dots\wedge\bar{z}_n.
\end{equation}

{\it We  fixe the orientation for which the real $2n$-vector
\eqref{VPOS} is positive} \cite{AWeil}. Introducing in \eqref{OMM} the
relation \eqref{nomega}, with
\eqref{VPOS},  we find 
\begin{equation}\label{OMMF}
\Omega_M=\mc{G} \dd V, \quad {\text{where}}\quad\dd V =2^n\dd  x_1\wedge \dd y_1\wedge
\dots\wedge \dd x_n\wedge \dd y_n.
\end{equation}
Let us  introduce the mapping $\Phi :\Hi^{\star}\!\rightarrow \fl$ (cf \cite{last,sb6}) 
\begin{equation}\label{aa}
\Phi(\psi):=f_{\psi},
f_{\psi}(z)=\Phi(\psi )(z)=(\varphi (z),\psi)_{\Hi}=(e_{\bar{z}},\psi)_{\Hi},
~z\in{\mathcal{V}}_0\subset M. 
\end{equation}
We  denote 
 by  $\fl:= L^2 _{\text{hol}} (M,\dd \nu_M)\cap\mc{O}(M)$  the Hilbert space of
 holomorphic, square integrable  functions  with respect to  the
 scalar product on $M$  given by the r.h.s. of \eqref{PAR}, 
\begin{equation}\label{scf}
(f,g)_{\fl} =\int_{M}\bar{f}(z)g(z)
\frac{\Omega_M(z,\bar{z})}{K_M(z,\bar{z})}
~=\int_{M}\bar{f}(z)g(z)\frac{\mc{G}}{K_M}\dd V,
\end{equation}
where $K_M: M\times \bar{M}\rightarrow\C$ admits the local   series
expansion in a base of orthonormal functions $\{\varphi_i\}_{i\in\N}$ with respect with the
scalar product \eqref{scf}, independent of the orthonormal base  (cf Theorem 2.1 in \cite{ber74}, based on  \cite{FUKs1})
\begin{equation}\label{funck}
K_M(z,\bar{w})\equiv (e_{\bar{z}},e_{\bar{w}}) =
\sum_{i=0}^{\infty}\varphi_i(z)\bar{\varphi}_i(w), \end{equation}
and
\begin{equation}\label{functphi}
\int_M\bar{\varphi}_i\varphi_j\frac{\Omega_M}{K_M}=\delta_{ij},~ i,j\in \N.
\end{equation}
{\it The function $K_M(z,\bar{w})$ is a reproducing kernel, i.e.  for
  $f\in\fl$, we
  have the relation} $$f(z)=(f,e_{\bar{z}})_{\fl}=\int_{M}(f,e_{\bar{w}})_{\fl}K_M(w,\bar{z})\frac{\Omega_M(w,\bar{w})}{K_M(w,\bar{w})}.$$ {\it  The symmetric
Fock space $\fl$ is the reproducing Hilbert space  with the scalar
product \eqref{scf}  attached to the 
kernel  function $K_M$ and the  evaluation map $\Phi$ defined by \eqref{aa}
 extends to an isometry} \cite{sb6} 
\begin{equation}\label{anti}
(\psi_1,\psi_2)_{\Hi^{\star}}=(\Phi (\psi_1),\Phi
(\psi_2))_{\fl}=(f_{\psi_{1}},f_{\psi_{2}})_{\fl}=
\int_M\overline{f}_{\psi_1} (z)f_{\psi_2}(z)\dd \nu_M(z) .
\end{equation}

In order to identify the Hilbert space $\Hi_f$ defined by \eqref{HIF}
with the Hilbert space with scalar product \eqref{PAR}, we have to
consider  the $\epsilon$-function  \cite{raw,Cah, cah}
\begin{equation}\label{EPSF}
\epsilon(z) = \e^{-f(z)}K_M(z,\bar{z}).
\end{equation}
If the K\"ahler metric on the complex manifold  $M$ is obtained by the
K\"ahler potential via \eqref{OMEGAMM} is such that
$\epsilon(z)$ is a positive  constant, then the metric is called {\it balanced}.
This denomination was firstly used  in \cite{don} for compact manifolds, then
it was used in \cite{arr} for noncompact manifolds and also in \cite{alo} in
the context of Berezin quantization on homogeneous bounded
domains. Note that 
{\it Condition } A) on p. 1132  {\it in Berezin's  paper} \cite{ber74}  {\it is exactly the
condition}  $\epsilon= ct$ in \eqref{EPSF}, where we have included the
Planck constant $h$ in the K\"ahler potential $f$.

In  \cite{sb6,sbj,FC} it was   considered the case
$\epsilon=1$, i.e.  we have considered in \eqref{OMEGAMM} $f=\ln
K$.

Under the condition $\epsilon=1$ in \eqref{EPSF}, the {\it hermitian
  balanced metric} of $M$ in
local coordinates  is obtained from the  scalar product   \eqref{scf} of two
CS vectors $(e_{\bar{z}},e_{\bar{z}}) =K_M(z,\bar{z})$ as 
\begin{equation}\label{herm}
\dd s^2_M(z)=\sum_{i,{j}=1}^nh_{i\bar{j}}\dd
z_{i}\otimes \dd\bar{z}_{j} =\sum_{i,{j}=1}^n \frac{\pa^2}{\pa
  z_{i} \pa\bar{z}_{j}} \ln (K_M(z,\bar{z}))
\dd
z_{i}\otimes \dd\bar{z}_{j} ,
\end{equation} 
with the associated K\"ahler two-form
\begin{equation}\label{hermo}
\omega_M(z)=i\sum_{i,{j}=1}^nh_{i\bar{j}}\dd
z_{i}\wedge \dd\bar{z}_{j} =\sum_{i,{j}=1}^n \frac{\pa^2}{\pa
  z_{i} \pa\bar{z}_{j}} \ln (K_M(z,\bar{z}))
\dd
z_{i}\wedge\dd\bar{z}_{j} .
\end{equation} 
We recall that in the footnote $*$ on p. 1128 in Berezin's paper \cite{ber74} it is
 mentioned that instead of the balanced K\"ahler  two-form  $\omega_M$ given by
 formula \eqref{hermo}, it should be alternatively used the (Bergman) K\"ahler two-form
\begin{equation}\label{omBERG} 
\omega^1_M=i\pa\bar{\pa}\ln(\mc{G})
,\end{equation} firstly used in the case of bounded
 domains and extended by Kobayashi \cite{koba} to a class of manifolds,   see \eqref{altom} in Proposition
 \ref{BIGCOR} in the Appendix. $\mc{G}$ is defined  in \eqref{nomega},
 \eqref{DETG}. In such a case, in \eqref{PAR} we should use instead
 of $\dd \nu_M$ given by \eqref{DELNU}, the expression $\dd
 \nu_M=(\omega^1_M)^n$.  Berezin has used in \cite{ber74}  the balanced  K\"ahler two-form
 \eqref{hermo} and not the Bergman two-form \eqref{omBERG} because he was
 able to  prove the 
 correspondence principle only with the balanced  metric given by \eqref{herm}.

Strictly speaking, the construction  of the 
Bergman kernel function  was advanced  for   bounded domains $\mc{D}\subset\C^n$. Firstly the case   $n=2$ was considered
in \cite{berg1,berg2,berg3}. In general, for complex manifolds $M$,
\eqref{herm}  defines a {\it pseudometric} (in the meaning of \cite{aisk}, see
also the Appendix). 
Berezin himself has applied his construction to quantization of $\C^n$ and
to the noncompact  hermitian symmetric spaces \cite{ber73, ber74,ber75},
realized as the  classical
domains I-IV in the formulation of Hua \cite{hua}. In the construction
of Berezin, a family of Hilbert spaces $\Hi_{\lambda}:=\Hi_{\lambda f}$ indexed by a
positive number $\lambda$ is considered, where the weight $\e^{-\lambda
  f}$ ($\lambda=\frac{1}{\hbar}$) is introduced  in formula \eqref{HIF} instead of $\e^{-f}$.

The relation \eqref{scf} can be interpreted in the language of
geometric quantization \cite{Kos}, see also \cite{SBS,csg}. The
local approach of Berezin to quantization via CS was
globalized by Rawnsley, Cahen and Gutt   in a series of papers
starting with   \cite{raw,Cah,cah}.
 In their approach,  the K\"ahler manifold $M$  is
not necessarily a homogeneous one, but they have proved that the 
homogeneous case corresponds to $\epsilon(z)= ct$ in \eqref{EPSF}.

Together with the K\"ahler manifold
$(M,\omega)$, it is also  considered the triple $\sigma =(\gl,h,\nabla)$, where
$\gl$ is a holomorphic line bundle on $M$, $h$ is the hermitian metric
on $\gl$ and $\nabla$ is a connection compatible with metric and the
K\"ahler structure \cite{SBS}. With respect to a local holomorphic
frame for the line  bundle, the metric can be given as
$$h(s_1,s_2)(z)=\hat{h}(z)\bar{\hat{s}}_1(z)\hat{s}_2(z),$$
where $\hat{s}_i$ is a local representing function for the section
$s_i$, $i=1, 2$,  and $\hat{h}(z)=(e_{\bar{z}},e_{\bar{z}})^{-1}$. \eqref{scf} is the local
representation of the scalar product on the line bundle
$\sigma$. The connection $\nabla$ has the expression  $\nabla=\pa +\pa \ln \hat{h} +\bar{\pa}$. The curvature of
$\gl$ is defined as
$F(X,Y)=\nabla_X\nabla_Y-\nabla_Y\nabla_X-\nabla_{[X,Y]}$, and locally
$F=\bar{\pa}\pa\ln\hat{h}$ \cite{chern}. The K\"ahler manifold  $(M,\omega)$  is
{\it quantizable} if there exists a triple $\sigma$ such that
$F(X,Y)=-i\omega(X,Y)$ and  we have \eqref{hermo}.

The extension of the construction of the  Bergman metric  on  bounded
domains in $\C^n$ \cite{berg1,berg2,berg3}  to K\"ahler manifolds  is a subtle one and was
considered by  Kobayashi \cite{koba,koba70}. In order to do this construction, instead of the Bergman
function,  Kobayashi has considered the {\it Bergman
  kernel form},  see  \eqref{berFORM} in the \S\ref{bsmb} in Appendix. 
 We also recall
that in the \S 4 of  Ch 4 of his  monograph \cite{ps1},  entitled  ``{\it Kobayashi 
  manifolds}'', Piatetski-Shapiro
used  this term
 in order to designate a  certain class of complex
manifolds studied by Kobayashi \cite{koba}, which have the characteristics of bounded domains in
$\C^n$. 

The square of the length of a vector $X\in\C^n$, measured in this
metric at the point $z\in M$,  is \cite{dinew}
$$\tau_M^2(z,X):=\sum_{i,j=1}^nh_{i\bar{j}}(z)X_{i}\bar{X}_{j} .$$
If the length $l$ of a piecewise $C^1$-curve $\gamma:[0,1]\ni t\mapsto
\gamma(t)\in M$ is defined as 
$$l(\gamma):=\int_0^1\tau_M(\gamma(t),\gamma'(t)) \dd t,$$
then the {\it Bergman distance} between two points $z_1,z_2\in M$ is
$$\dd_B(z_1,z_2)=\text{inf}\left\{ l(\gamma):\gamma ~\text{is a
      piecewise curve s.t.} ~\gamma(0)=z_1,\gamma(1)=z_2\right\}.$$

We  denote  the {\it{ normalized Bergman kernel}}  ({\it{the two-point
function}} of $M$ \cite{ber97,SBS})   by 
\begin{equation}\label{kmic}
\kappa_M(z,\bar{z}'):=\frac{K_M(z,\bar{z}')}{\sqrt{K_M(z)K_M(z')}}=
(\tilde{e}_{\bar{z}},\tilde{e}_{\bar{z}'})=\frac{(e_{\bar{z}},e_{\bar{z}'})}{\|e_{\bar{z}}\|\|e_{\bar{z}'}\|} .
\end{equation}
Introducing in the above definition  the series expansion
\eqref{sumBERG}, with  the Cauchy - Schwartz inequality, we have that 
\begin{equation}|\kappa_M(z,\bar{z}')|\le 1, 
\text{{{~~~~~~and~~~~~~}}} \kappa_M(z,\bar{z})=1. \end{equation}
In this paper,  by  {\it  the Berezin kernel} $b_M:M\times M\rightarrow
[0,1]\in \R$ we mean:
\begin{equation}\label{berK}
b_M(z,z'):=|\kappa_M(z,\bar{z}')|^2.
\end{equation}
Note that 
\begin{equation}\label{DIA}D_M(z,z'):= -\ln b_M(z,z') = -2\ln
\left\vert(\tilde{e}_{\bar{z}},\tilde{e}_{\bar{z}'})\right\vert
\end{equation}
is {\it Calabi's diastasis} \cite{calabi} expressed via the  CS
vectors \cite{cah}.

Let $\xi:\Hi\setminus 0\rightarrow\db{P}(\Hi)$ be the  the  canonical projection 
$\xi(\mb{z})=[\mb{z}]$. The {\it Fubini-Study metric}  in the
nonhomogeneous coordinates $[z]$ is  the
hermitian metric on $\db{CP}^{\infty}$ (see \cite{koba} for details) 
\begin{equation}\label{FBST}\dd s^2|_{FS}(\mb[{z}])=
\frac{(\dd\mb{z},\dd\mb{z}) (\mb{z},\mb{z})-(\dd\mb{z},\mb{z})
  (\mb{z},\dd\mb{z})}{(\mb{z},\mb{z})^2}.
\end{equation}

The elliptic {\it Cayley distance}    \cite{Cay}  between two points in the projective Hilbert space
$\db{P}(\Hi)$ is defined as 
\begin{equation}\label{Cdis}
\dd_C([z_1],[z_2])=\arccos \frac{|(z_1,z_2)|}{\|z_1\| 
\|z_2 \| }.
\end{equation}
The  Fubini-Study metric \eqref{FBST} and the Cayley distance
\eqref{Cdis} are  independent of the homogeneous coordinates $z$
representing $[z]=\xi(z)$.

Let $M$ be a homogeneous K\"ahler manifold $M=G/H$ to which we
associate the Hilbert space of functions \fl~ with respect to the scalar
product \eqref{scf}.  We consider manifolds $M$  which are  CS-{\em
  orbits}, i.e.  which admit a holomorphic
embedding   $\iota_M : M \hookrightarrow \Phinf$ \cite{lis,neeb,sb6}.
Note that this embedding differs from the standard Kobayashi embedding
recalled in Theorem \ref{Bigth} in the Appendix.  

For  the following assertions, see  \cite{csg}:
\begin{Remark}\label{HTR}Let us suppose that the K\"ahler manifold $M$ admits a holomorphic
embedding 
\begin{equation}\label{invers}
\iota_M: M\hookrightarrow \db{CP}^{\infty},~~ \iota_M(z) =  [\varphi_0(z):\varphi_1(z):\dots ] .
\end{equation}
The balanced Hermitian metric \eqref{herm} on $M$ is
the pullback of the Fubini-Study metric \eqref{FBST} via the embedding
\eqref{invers}, i.e.:
\begin{equation}\label{KOL} 
\dd s^2_M(z)=\iota_M^*\dd s^2_{FS}(z)= \dd s^2_{FS}(\iota_M(z)).
\end{equation}  
The angle defined by the normalized Bergman kernel \eqref{kmic} can
be expressed via the embedding \eqref{invers}  as function of the Cayley
distance \eqref{Cdis}
\begin{equation}\label{c1}
\theta_M(z_1,z_2)=\arccos|\kappa_M(z_1,\bar{z}_2)|=
\arccos|(\tilde{e}_{z_1},\tilde{e}_{z_2})_M|
=d_C(\iota_M(z_1),\iota_M(z_2)).
\end{equation}

We have also the relation 
\begin{equation}\label{ddbm}\dd_B(z_1,z_2)\ge\theta_M(z_1,z_2) . 
\end{equation}
The following (Cauchy) formula is true
\begin{equation}\label{c2}
 (\tilde{e}_{z_1},\tilde{e}_{z_2})_M=(\iota_M(z_1),\iota_M(z_2))_{\db{CP}^{\infty}}.
\end{equation}
The Berezin kernel \eqref{berK} admits the geometric interpretation via the Cayley
distance as
\begin{equation}
b_M(z_1,z_2)=\cos^2d_C(\iota_M(z_1),\iota_M(z_2))=\frac{1+\cos(2d_C(\iota_M(z_1),\iota_M(z_2)))}{2} .
\end{equation}
\end{Remark}
Note that traditionally ``the Kobayashi embedding'' is realized usually
by the metric \eqref{kobaM}  obtained using the top-degree holomorphic
forms on $M$, see the Appendix.
\section{The geometry of the Siegel-Jacobi disk}\label{TSJD}
The Siegel-Jacobi disk is the 4-dimensional homogenous space 
\begin{equation}\label{mm}
\mc{D}^J_1:= H_1/\R\times
 \text{SU}(1,1)/\text{U}(1)= \C\times\mc{D}_1,
\end{equation}
 associated to the 6-dimensional
Jacobi group $G^J_1=H_1\ltimes\text{SU}(1,1)$, where $H_1$ is the
3-dimensional Heisenberg group,  and  the Siegel disk $\mc{D}_1$ is
realized as
\begin{equation}\label{SDISK}
\mc{D}_1=\{w\in\C^1 ~~\vert  ~~ |w|<1\}.
\end{equation}
It is easy to see \cite{jac1} that:
\begin{Remark}\label{rempeste}
The Bergman  kernel function  $K_k:\mc{D}_1\times\bar{\mc{D}}_1\rightarrow \C$ is
\begin{equation}\label{ker2}
K _k (w,\bar{w}'):=
(1-w\bar{w}')^{-2k}=\sum_n^{\infty} f_{nk}(w)\bar{f}_{nk}(w'),
\end{equation}
where 
\begin{equation}\label{f2}
f_{nk}(w):=
\sqrt{\frac{\Gamma (n+2k)}{n!\Gamma (2k)}}w^n, \quad w\in\mc{D}_1,
\end{equation}
and the Siegel disk $\mc{D}_1$ has the  K\"ahler two-form $\omega_k$
given by
\begin{equation}\label{mdisk}
-i \omega_k(w) =\frac{2k}{(1-w\bar{w})^2}\dd w\wedge \dd
\bar{w}, 
\end{equation}
${\emph{\text{SU}}}(1,1)$-invariant to the linear fractional  transformation 
\begin{equation}\label{dg} w_1= g\cdot w 
=\frac{a \, w+ b}{\delta}, ~ \delta=\bar{b}w+\bar{a},
{\emph{\text{SU}}}(1,1) \ni g= \left( \begin{array}{cc}a & b\\ \bar{b} &
\bar{a}\end{array}\right),~~ ~|a|^2-|b|^2=1.
\end{equation} 

The action \eqref{dg} is   transitive, and   the Siegel disk
$\mc{D}_1$ is a  K\"ahler
  homogeneous manifold.   $\mc{D}_1$  is  a symmetric space.
\end{Remark}
Let us introduce the notation $\mc{D}^J_1\ni\varsigma:=(z,w)\in\C\times\mc{D}_1$.
Following the methods of \cite{jac1}, we get  the reproducing kernel
function 
$K:\mc{D}^J_1\times  \mc{D}^J_1\rightarrow\C$,  $K_{k\mu}(\varsigma,\varsigma')
  :=  (e_{\bar{z},\bar{w}},e_{\bar{z}',\bar{w}'})$, obtained from the
  scalar product of two  CS  vectors  of the type $e_{z,w}$,  where $k=k'+\frac{1}{4}$,
   $2k'\in\Z$ and $\mu\in\R_+$:
\begin{equation}\label{ker3}
K_{k\mu}(\varsigma,\bar{\varsigma}')=
(1-{w}\bar{w}')^{-2k}\exp{\mu F(\varsigma,{\varsigma}')}, ~
F(\varsigma,\bar{\varsigma}') = 
\frac{2\bar{z}'{z}+z^2\bar{w}'+\bar{z}'^2w}{2(1-{w}\bar{w}')}.
\end{equation}
In particular, the kernel on the  diagonal, 
$K_{k\mu}(\varsigma)=(e_{\bar{z},\bar{w}},e_{\bar{z},\bar{w}})$, reads
\begin{equation}\label{hot}
K_{ k\mu}(\varsigma)=
(1-w\bar{w})^{-2k}\exp{\mu F(z,w)},~ F(z,w)=\frac{2z\bar{z}+z^2\bar{w}+\bar{z}^2w}{2(1-w\bar{w})}, 
\end{equation}
and evidently $K_{ k\mu}(\varsigma) >0,~ \forall \varsigma\in\mc{D}^J_1.$
The holomorphic, transitive and effective  
action of the Jacobi  group $G^J_1$
on the manifold $\mc{D}^J_1$ is
$(g,\alpha)\circ (z,w)\rightarrow(z_1,w_1)$, where $w_1$ is given by
\eqref{dg} and 
\begin{equation}\label{xxx}
z_1=\frac{\gamma}{\delta}, ~\gamma = z
+\alpha-\bar{\alpha}w.\end{equation}

\subsection{The symmetric Fock space}\label{SFS}
The scalar product \eqref{PAR} of 
functions from the space $\got{F}_{k\mu}=L^2_{\text{hol}}(\mc{D}^J_1,\rho_{k\mu})$
corresponding to the kernel $K_{k\mu}$ 
defined by \eqref{hot}  on the manifold (\ref{mm}) is \cite{jac1,csg}:
\begin{equation}\label{ofi}
(\phi ,\psi )_{k\mu}= \int_{\mc{D}^J_1}\!\bar{f}_{\phi}(z,w)f_{\psi}(z,w)\rho_{k\mu}
\dd \nu(z,w) ,\rho_{k\mu} = \Lambda\! (1\!-\!w\bar{w})^{2k}\!\e^{\!-\mu\frac{2|z|^2+z^2\bar{w}\!+\!\bar{z}^2w}{2(1\!-\!w\bar{w})}},
\end{equation}
\begin{equation}\label{ofi1}
\Lambda = \frac{4k-3}{2\pi^2}.
\end{equation}
The value of the $G^J_1$-invariant measure \eqref{DELNU},   obtained
in  (\ref{dnu}), is   
\begin{equation}\label{ofi3}
\dd \nu(z,w) =\mu\frac{\dd \Re w \dd\Im w}{(1-w\bar{w})^3}\dd \Re z \dd \Im z.
\end{equation}

The base of orthonormal functions attached  to  kernel \eqref{ker3} 
 defined  on the manifold $\mc{D}^J_1$ consists of  the holomorphic polynomials 
\cite{jac1}, \cite{sbcg} 
\begin{equation}\label{x4x}
f_{nk'm}(z,w)=f_{mk'}(w)\frac{P_n(\sqrt{\mu}z,w)}{\sqrt{n!}}, ~k=k'+\frac{1}{4},~2k'\in \Z_+,~\mu\in\Z_+,
\end{equation}
where the monomials  $f_{mk'}$ are defined in \eqref{f2},
while   
\begin{equation}\label{marea}
P_n(z,w)=n!\sum _{p=0}^{[\frac{n}{2}]}
(\frac{w}{2})^p\frac{z^{n-2p}}{p!(n-2p)!} .
\end{equation}
 
The series expansion \eqref{sumBERG}  of the   kernel function  \eqref{ker3} 
 reads
\begin{equation}
\label{ker31}K_{k\mu}(z,w;\bar{z}',\bar{w}')   =
\sum_{n,m=0}^{\infty}f_{nk'm}(z,w)\bar{f}_{nk'm}(z',w').
\end{equation}
We may also write down  the expression \eqref{ker31} as
\begin{equation}
\label{ker32}K_{k\mu}(z,w;\bar{z}',\bar{w}')   =
\sum_{n,m=0}^{\infty}\tilde{f}_{nkm}(z,w)\bar{\tilde{f}}_{nkm}(z',w'),
\end{equation}
where
\begin{equation}\label{ker34}
\tilde{f}_{nkm}(z,w)=a_{nkm}w^mP_n(\sqrt{\mu}z,w),
~~  a_{nkm}=\sqrt{\frac{\Gamma(m+2k-1)}{m!\Gamma(2k-1)n!}},
\end{equation}
and the orthonormality of the base can be written as (see calculation
in \cite{ber9})
\begin{equation}\label{orthJ}
(\tilde{f}_{nks},\tilde{f}_{mkr})_{k\mu}=\delta_{nm}\delta_{rs},~~
n,m,r,s\in\N,~ k>\frac{1}{2}.
\end{equation}

\subsection{Two-forms}\label{TWOf}
We  recall the  definitions of the  notions which appear in Proposition \ref{prop2}.

The {\it  Ricci form associated to the  K\"ahlerian two-form $\omega_M$} \eqref{kall}  is  (see
  p. 90 in  \cite{mor}) 
\begin{equation}\label{RICCI}
\rho_M(z):=i \sum_{\alpha,\beta=1}^n\text{Ric}_{\alpha\bar{\beta}}(z)\dd
z_{\alpha}\wedge\dd \bar{z}_{\beta}, ~ \text{Ric}_{\alpha\bar{\beta}}(z)=
 -\frac{\pa^2}{\pa z_{\alpha}\pa \bar{z} _{\beta}} \ln\mc{G}(z).
\end{equation}

The scalar curvature at a point $p\in M$ of coordinates $z$ is (see
p. 294 in \cite{koba1})
\begin{equation}\label{scc}
s_M(p):=\sum_{\alpha,{\beta}=1}^n
(h_{\alpha{\bar\beta}})^{-1}\text{Ric}_{\alpha\bar{\beta}}(z). 
\end{equation}
We use the following expression for the Laplace-Beltrami operator on
K\"ahler manifolds $M$ with the K\"ahler two-form  \eqref{kall}, cf e.g. Lemma 3
in the Appendix of \cite{ber74}:
\begin{equation}\label{LPB}
\Delta_M(z):=\sum_{\alpha,\beta=1}^n(h_{\alpha\bar{\beta}})^{-1}\frac{\pa^2}{\pa \bar{z}_{\alpha}\pa{{z}_{\beta}}}.
\end{equation}

Following \cite{koba,LU08},  let us introduce also the positive definite
(1,1)-form on $M=G/H$
\begin{equation}\label{RIC2}\tilde{\omega}_M(z):= 
i\sum_{\alpha,\beta\in\Delta_+} \tilde{h}_{\alpha\bar{\beta}} (z) \dd z_{\alpha}\wedge
\dd\bar{z}_{\beta}, ~  \tilde{h}_{\alpha\bar{\beta}} (z):= 
  (n+1)h_{\alpha\bar{\beta}} (z)-  \text{Ric}_{\alpha\bar{\beta}}(z),
\end{equation}
which is K\"ahler, with K\"ahler potential $$\tilde{f}:=\ln
(K(z)^{n+1}\mc{G}(z)).$$

\begin{Proposition}\label{prop2}
The balanced  K\"ahler two-form $\omega_{k\mu}$ on $\mc{D}^J_1$, $G^J_1$-invariant to the
action \eqref{dg}, \eqref{xxx}, can be written as:
\begin{equation}\label{aab}
-i\omega_{k\mu}(z,w) =2k\frac{\dd w \wedge \dd\bar{w}}{P^2} +
\mu\frac{\mc{A}\wedge \bar{\mc{A}}}{P},
~\mc{A}=\dd z+\bar{\eta}(z,w)\dd w, ~\eta(z,w)=\frac{z+\bar{z}w}{P},
\end{equation}
where $$P=1-w\bar{w}.$$
The balanced  hermitian metric on $\mc{D}^J_1$ corresponding to the K\"ahler two-form
\eqref{aab}
is:
 \begin{equation}\label{metric}
\dd s^2_{k\mu}(z,w)=2k\frac{\dd w \otimes \dd\bar{w}}{P^2} +
\mu\frac{\mc{A}\otimes\bar{\mc{A}}}{P}.
\end{equation}

The volume form  is:
\begin{equation}\label{dnu}
\omega_{k\mu}\wedge\omega_{k\mu} = 16k\mu (P)^{-3} \dd \Re z \wedge \dd \Im z \wedge\dd
 \Re w  \wedge \dd \Im w  ,
\end{equation}
giving for the density of the  $G^J_1$-invariant volume $\mc{G}$
\eqref{DETG} of  Siegel-Jacobi disk a value independent of $z$ \begin{equation}\mc{G}_{\mc{D}^J_1}(z,w)=\frac{2k\mu}{(1-w\bar{w})^3}.\end{equation} 
The Bergman metric \eqref{TRUB} is
\begin{equation}\label{bmetric}
\dd s^2_{\mc{B}}(z,w)_{\mc{D}^J_1}=3\frac{\dd w\otimes \dd \bar{w}}{(1-w\bar{w})^2},
\end{equation}
with the associated Bergman K\"ahler two-form \eqref{altom}
\begin{equation}\label{newkal}
\omega^1_{\mc{D}^J_1}(z,w)= 3i \frac{\dd w\wedge \dd \bar{w}}{(1-w\bar{w})^2}.
\end{equation}
The K\"ahler two-form \eqref{RIC2} for $\mc{D}^J_1$, corresponding to
the K\"ahler potential $$\tilde{f}(z,w)= 3[\mu f(z,w)-(2k+1)\ln
(1-w\bar{w})], $$ reads
\begin{equation}\label{doiD}
-i \tilde{\omega}_{\mc{D}^J_1} (z,w)= 
3 \left[(2k+1)\frac{\dd w \wedge \dd\bar{w}}{P^2} +
\mu\frac{\mc{A}\wedge \bar{\mc{A}}}{P}\right].
\end{equation}
The Ricci form \eqref{RICCI}  associated with the balanced metric
\eqref{metric} reads
\begin{equation}\label{RICJD}
\rho_{\mc{D}^J_1}(z,w) = -3i\frac{\dd w\wedge\dd \bar{w}}{(1-w\bar{w})^2},
\end{equation}
and  $\mc{D}^J_1$   is not an Einstein manifold with respect to the
balanced metric \eqref{metric}, but it is one with respect to  the
Bergman metric \eqref{bmetric}.

 The scalar curvature \eqref{scc} is  constant and    negative definite
\begin{equation}\label{sCAL}
s_{\mc{D}^J_1}(z,w)=-\frac{3}{2k}, ~~p\in \mc{D}^J_1.\end{equation}
The Laplace-Beltrami operator \eqref{LPB}  on the
Siegel-Jacobi disk has the
expression
\begin{equation}\label{LBD}
\begin{split}
\Delta_{\mc{D}^J_1}(z,w) & =\frac{P^2}{2k\mu}\left[(\frac{2k}{P}+\mu|\eta|^2)
\frac{\pa^2}{\pa z\pa\bar{z}} +
\mu (\frac{\pa^2}{\pa w \pa\bar{w}}-
\bar{\eta}\frac{\pa^2}{\pa z\pa\bar{w}}-{\eta}\frac{\pa^2}{\pa\bar{z}\pa
  w})\right]\\
~~~~~~~& = \left(
  \frac{1-w\bar{w}}{\mu}+\frac{|z+\bar{z}w|^2}{2k}\right)\frac{\pa^2}{\pa z\pa\bar{z}} 
+\\
~~~~~~~& +
\frac{1-w\bar{w}}{2k}\left[ (1-w\bar{w}) \frac{\pa^2}{\pa w
    \pa\bar{w}}- (\bar{z}+{z}\bar{w})\frac{\pa^2}{\pa z\pa\bar{w}} -({z}+\bar{z}{w})\frac{\pa^2}{\pa\bar{z}\pa
  w} \right] ,
\end{split}
\end{equation}
$G^J_1$-invariant to the action \eqref{dg}, \eqref{xxx}, i.e. 
\begin{equation}\label{invar}
\Delta_{\mc{D}^J_1}(z_1,w_1)= \Delta_{\mc{D}^J_1}(z,w).
\end{equation}
Also we have the relations:
\begin{equation}\label{estiut}
\Delta_{\mc{D}^J_1}(z,w)(\ln(\mc{G}_{\mc{D}^J_1}(z,w)))=-s_{\mc{D}^J_1}(z,w)=\frac{3}{2k}.
\end{equation}
\end{Proposition}
\begin{proof}
Most of the assertions  of  Proposition \ref{prop2}  have been  proved in
\cite{jac1,csg}. Here we emphasize some differences between the balanced
and Bergman metric on the Siegel-Jacobi disk. However, for
self-containment, we indicate the main ingredients of the proof.

We calculate the K\"ahler potential on $\mc{D}^J_1$ 
as the logarithm of the reproducing kernel 
$f(z,w)=\ln K_{k\mu}(z,w)$,  $z,w\in\C$, $|w|<1$, 
i.e.
\begin{equation}\label{keler}
f(z,w) =\mu\frac{2z\bar{z}+z^2\bar{w}+\bar{z}^2w}{2(1-w\bar{w})}
-2k\ln (1-w\bar{w}) .
\end{equation}
The balanced K\"ahler two-form $\omega$ is obtained  with 
 formulas \eqref{kall}, \eqref{Ter}, i.e. 
\begin{equation}\label{aaa}
-i \omega_{k\mu}(z,w) = h_{z\bar{z}}dz\wedge d\bar{z}+h_{z\bar{w}}dz\wedge
d\bar{w}
-h_{\bar{z}w}d\bar{z}\wedge d w +h_{w\bar{w}}dw\wedge d\bar{w} .
\end{equation}
The volume form \eqref{OMM} for $\mc{D}^J_1$ is:
\begin{equation}\label{vol1}-\omega\wedge\omega=2\left|\begin{array}{cc}
h_{z\bar{z}} & h_{z\bar{w}}\\
h_{\bar{z}w} & h_{\bar{w}w}
\end{array}\right|
\dd z\wedge \dd \bar{z}\wedge \dd w \wedge \dd \bar{w} .
\end{equation}
The matrix of the balanced metric  \eqref{herm} $h=h(\varsigma)$, 
$\varsigma:=(z,w)\in\C\times\mc{D}_1$, determined with
the K\"ahler potential \eqref{keler}, reads
\begin{equation}\label{metrica}
h(\varsigma):=\left(\begin{array}{cc}h_{z\bar{z}} & h_{z\bar{w}}\\
    \bar{h}_{z\bar{w}} & h_{w\bar{w}}\end{array}\right) =
\left(\begin{array}{cc} \frac{\mu}{P} & \mu \frac{\eta}{P} \\
\mu\frac{\bar{\eta}}{P} &
\frac{2k}{P^2}+\mu\frac{|\eta|^2}{P}\end{array}\right),  
\end{equation}
where $\eta$ is  defined  as in \cite{jac1}
\begin{equation}\label{csv}
 ~z=\eta-w\bar{\eta},\quad\text{and~~~}~~~\eta = \eta(z,w):= \frac{z+\bar{z}w}{1-w\bar{w}}.
\end{equation}

The inverse of the matrix \eqref{metrica} reads
\begin{equation}\label{hinv}
h^{-1}(\varsigma)=\frac{P^3}{2k\mu}\left(\begin{array}{cc}\frac{2k}{P^2}+\mu\frac{|\eta|^2}{P}
    & - \mu \frac{\eta}{P} \\ -\mu\frac{\bar{\eta}}{P} & \frac{\mu}{P} \end{array}\right) .
\end{equation}

The check out   the $G^J_1$-invariance of the Laplace-Beltrami operator
\eqref{LBD} is an easy calculation, but quite long. We indicate some
intermediate stages. 

The inverse  of the relations \eqref{dg}, \eqref{xxx} are 
\begin{equation}\label{acinv}
w=\frac{\bar{a}w_1-b}{-\bar{b}w_1+a},\quad z=\frac{-z_1+\alpha a +\bar{\alpha}b-w_1(\alpha\bar{b}+\bar{\alpha}\bar{a})}{{-\bar{b}w_1+a}},
\end{equation}
where $g\in \text{SU}(1,1)$, defined by \eqref{dg}, and $\alpha\in\C$ define the action of
the Jacobi group on the Siegel-Jacobi disk.
We have$$\frac{\pa}{\pa z_1}= \frac{\pa z}{\pa z_1} \frac{\pa}{\pa z} ,\quad
\frac{\pa}{\pa w_1}= \frac{\pa w}{\pa w_1}\frac{\pa}{\pa w}+ \frac{\pa
z}{\pa w_1}\frac{\pa }{\pa z},$$
where, with \eqref{dg}, \eqref{xxx}, \eqref{acinv},  we find:
$$ \frac{\pa z}{\pa z_1}=\delta;\quad \frac{\pa w}{\pa
  w_1}=\delta^2;\quad \frac{\pa z}{\pa w_1}=\delta{\vartheta},~~~ \vartheta =
\bar{\alpha}\bar{a}+\bar{b}(z+\alpha).$$
\begin{equation}\label{derP}
\begin{split}
\frac{\pa^2}{\pa z_1 \pa  \bar{z}_1} &= |\delta|^2\frac{\pa^2}{\pa  z\pa \bar{z}} ,\\
\frac{\pa^2}{\pa z_1 \pa \bar{w}_1} &=
|\delta|^2(\bar{\delta}\frac{\pa^2}{\pa  z\pa \bar{w}} +\bar{\vartheta}
\frac{\pa^2}{\pa  z\pa \bar{z}}  ),\\
\frac{\pa^2}{\pa w _1 \pa \bar{w}_1} &=|\delta|^2( |\delta|^2s
\frac{\pa^2}{\pa w  \pa \bar{w}}+\bar{\delta}\vartheta \frac{\pa^2}{\pa z
  \pa \bar{w}}+ {\delta}\bar{\vartheta} \frac{\pa^2}{\pa \bar{z}
  \pa {w}} +|\vartheta|^2 \frac{\pa^2}{\pa {z}
  \pa \bar{z}} ).
\end{split}
\end{equation}
We also find easily
\begin{equation}\label{otherP}
\begin{split}
P' &=  1-w_1\bar{w}_1=\frac{P}{|\delta|^2},\\
\eta_1 &=\frac{z_1+\bar{z}_1w_1}{P'}=a(\eta+\alpha)+b(\bar{\eta}+\bar{\alpha}).
\end{split}
\end{equation}
With \eqref{derP}, \eqref{otherP}, we check out  the invariance
\eqref{invar}  of the
Laplace-Beltrami operator   \eqref{LBD}  on the Siegel-Jacobi disk to the action of
the Jacobi group $G^J_1$.  
\end{proof}
In Remark 2 in
\cite{nou} we have already stressed  that
the K\"ahler two-form   \eqref{aab}, firstly calculated  in \cite{jac1},
is  identical with
the one  obtained by J.-H. Yang, see e.g. Theorem 1.3 in \cite{Y10},
where $A,B, w,\eta$ corresponds in our notation with respectively
$\frac{1}{2}k,\frac{1}{4}\mu ,-w,z$. Under the same correspondence,  the formula
\eqref{LBD} is a particular case of Theorem 1.4 in  \cite{Y10}. The scalar curvature  \eqref{sCAL} was previously
obtained in \cite{jae}. We recall that Theorem 2.5 in  Berezin's paper
\cite{ber74}
asserts essentially that $\Delta_M(z)(\ln(\mc{G}(z)))= ct $  for the balanced metric  plus other 3 conditions. 

\subsection{Embeddings}\label{EMSS}

We recall that the homogeneous K\"ahler manifolds  $M=G/H$  which admit an
embedding in a  projective Hilbert space  as in  Remark \ref{HTR} are called
{\it CS-manifolds}, and the corresponding  groups $G$ are called
{\it CS-groups} \cite{lis,neeb}. We particularize Remark
\ref{HTR} in the case of the Siegel-Jacobi disk and we have: 
\begin{Proposition} \label{REM9}
The Jacobi group $G^J_1$
is a CS-group and the Siegel-Jacobi disk  $\mc{D}^J_1$ is a quantizable K\"{a}hler
 CS-manifold. The Hilbert space of functions 
\fl ~ is the space $ \got{F}_{k\mu}=
L^2_{\emph{\text{hol}}}(\mc{D}^J_1,\rho_{k\mu})$ with  the scalar product
\eqref{ofi}-\eqref{ofi3}. 
The   K\"ahlerian embedding  $\iota_{\mc{D}^J_1}:\mc{D}^J_1\hookrightarrow
\db{CP}^{\infty}$  \eqref{invers}  $\iota_{\mc{D}^J_1} =[\Phi]=
[\varphi_0:\varphi_1:\dots\varphi_N: \dots ] $ is realized with  an ordered version of the base functions
$\Phi =\left\{ f_{nk'm}(z,w) \right\}$  given by \eqref{x4x}, 
 and the K\"ahler two-form  \eqref{aab} is the pullback of
the Fubini-Study K\"ahler two-form \eqref{FBST} on $\db{CP}^{\infty}$,
$$\omega_{k\mu}= \iota_{\mc{D}^J_1}^*\omega_{FS}|_{\db{CP}^{\infty}},
~\omega_{k\mu}(z,w)=\omega_{FS}([\varphi_N(z,w)]).$$ 
The normalized Bergman kernel \eqref{kmic}  $\kappa_{k\mu}$  of the Siegel-Jacobi disk 
expressed in the variables $\varsigma=(z,w)$, $\varsigma'=(z',w')$ reads
\begin{equation}\label{redBKJ}
\kappa_{k\mu}(\varsigma,\bar{\varsigma}')=\kappa_{k}(w,\bar{w}')\exp[\mu
(F(\varsigma,\bar{\varsigma}')-\frac{1}{2}(F(\varsigma)+F(\varsigma'))], 
\end{equation}
where $\kappa_{k}(w,\bar{w}')$ is the normalized Bergman kernel for
the Siegel disk $\mc{D}_1$
\begin{equation}
\kappa_{k}(w,\bar{w}')=\left[\frac{(1-|w|^2)(1-|w'|^2)}{(1-w\bar{w}')^2}\right]^k,
\end{equation}
 $F(\varsigma,\bar{\varsigma}')$ is defined in \eqref{ker3}, and
$F(\varsigma)$ is defined in \eqref{hot}. 
The Berezin  kernel of $\mc{D}^J_1$ is 
$$b_{k\mu}(\varsigma,\varsigma')=b_{k}(w,w')\exp[2\Re
F(\varsigma,\bar{\varsigma}')-F(\varsigma)- F(\varsigma')],$$
where $b_{k}(w,w')=|\kappa_{k}(w,\bar{w}')|^2$. 

With formula \eqref{DIA}, we get for the
diastasis function on the Siegel-Jacobi disk the expression:
\begin{equation}\label{dia1}
\frac{D_{k\mu}(\varsigma,\varsigma')}{2}=k
  \ln\frac{\vert 1-w\bar{w}'\vert^2}{(1-|w|^2)(1-|w'|^2)} +\mu[\frac{F(\varsigma)+F(\varsigma')}{2}-\Re F(\varsigma,\bar{\varsigma'})].
\end{equation}
\end{Proposition}
\begin{proof}The above proposition was enunciated in  \cite{csg}. Here we use the
notion of embedding as in the standard definition recalled at the
beginning of  \S \ref{EMBB}. This approach is different from the standard Kobayashi embedding
summarized in \S \ref{KEM} of the Appendix.  We use the explicit
representation \eqref{ker32}. 
In accord with \eqref{ker34}, we have $$\tilde{f}_{0k0}=1, ~\tilde{f}_{0k1}=\sqrt{2k-1}w,
~\tilde{f}_{1k0}=\sqrt{\mu}z, $$
$$\frac{\pa(\tilde{f}_{0k1},\tilde{f}_{1k0})}{\pa(z,w)}=-\sqrt{\mu(2k-1)}\not=
0, ~~~
k>\frac{1}{2},~\mu>0,$$
and we get 
$$K_{k\mu}(z,\bar{w})=1+(2k-1)|w|^2+\mu|z|^2+\sum_{n,s\in\db{G}}|\tilde{f}_{nks}(z,w)|^2,$$
where  $\db{G}=\N\times\N \setminus\{(0,1)\cup(1,0)\}$.
\end{proof}
\section{Bergman representative coordinates on homogeneous K\"ahler manifolds}\label{BRC}
\subsection{Definition}Bergman  has introduced the representative
coordinates on   bounded domains  $\mc{D}
\subset \C^n$ \cite{berg} in
order to generalize the Riemann mapping theorem to $\C^n$, $n>1$. 
However, almost the same construction works in
the case of manifolds $M$ \cite{davi,dinew} instead of bounded domains. Berezin's approach to
quantization  \cite{ber73,ber74,ber75,berezin}
recalled  in \S \ref{gi}  was  applied to manifolds $M$ which are (symmetric)
bounded domains $\mc{D}\subset\C^n$ (in fact, hermitian symmetric
spaces) and  $\C^n$. If   the same construction is applied  to
 manifolds $M$ which are  not necessarily bounded domains, then the Hilbert
space $\fl$ \eqref{scf} usually is  replaced with the Hilbert space of
square integrable global holomorphic forms of top degree  $\gf_n(M)$
defined  by \eqref{sc3} in the Appendix. 
 In the presentation below of the Bergman representative coordinates
 we shall use  the notation for bounded  domains
 \cite{dinew} adapted for manifolds $M$.

Assume now that \eqref{hermo} defines a K\"ahlerian structure on the
complex $n$-dimensional manifold $M$. Then 
\begin{equation}\label{cdnon0}
\mc{G}(z)=\det h_{i,\bar{j}}(z)=\det \frac{\pa^2\ln( K_M(z,\bar{z}))}{\pa
  z_i{\pa}\bar{z}_j }\not= 0, \quad i,j=1,\dots,n , 
\end{equation}
and in a neighborhood
of $\zeta\in M$, the holomorphic functions of $z$
\begin{equation}\label{mumu}
\mu_j(z) = \frac{\pa}{\pa\bar{\zeta}_j}
\ln \frac{K_M(z,\bar{\zeta})}{K_M(\zeta,\bar{\zeta})},~~~~  j=1,\dots,n,\end{equation} form  a local
system of coordinates, due to the condition \eqref{cdnon0}. If we
consider the mapping $f:M\rightarrow M$ given by a holomorphic transformation
$(z,\zeta)\rightarrow (\tilde{z},\tilde{\zeta})$, then we have 
$$\mu_j(z)=\sum_{k=1}^n\tilde{\mu}_k(\tilde{z})\frac{\pa{\bar{{\tilde{\zeta}}}_k}}{\pa\bar{\zeta}_j}, ~~j=1,\dots,n,$$i.e. {\it any biholomorphic transformation of $M$ can be described
  as a linear transformation of the coordinates} $\mu_j$, called {\it
    covariant representative coordinates in} $M $, see e.g. \cite{mskw}. 
The following linear combinations of the coordinates $\mu_j$
\eqref{mumu}  are  considered
$$\left(\begin{array}{c}w_1\\ \vdots\\
    w_n\end{array}\right)=h^{-1}\left(\begin{array}{c}\mu_1\\ \vdots\\
    \mu_n\end{array}\right) .$$

{\it The Bergman representative coordinates relative to a point
$z_0$ of a homogenous K\"ahler manifold $M$} are defined by the formulae
\begin{equation}\label{brc}
w_i(z)=\sum_{j=1}^n h^{\bar{j}i}(z_0)\frac{\pa}{\pa\bar{\zeta}_j}
\ln \frac{K_M(z,\bar{\zeta})}{K_M(\zeta,\bar{\zeta})}\Bigr|_{\zeta=z_0},~  i=1,\dots,n,
\end{equation}
where $h^{\bar{j}i}(z_0)$ is the inverse of the matrix
$h_{i\bar{j}}(z_0)$ calculated with \eqref{herm} from the 
kernel function  $K_M$ \eqref{funck}. This definition differs from the standard
definition of the Bergman representative coordinates, where, instead of
the kernel function $K_M$ \eqref{funck}, it is used the Bergman kernel
of the complex manifold $\mc{B}_n(z,\bar{w})$ defined in
\eqref{FForm} and instead of the matrix $h$ of the balanced metric
\eqref{herm} it is used the Bergman tensor  \eqref{TTT}. 

The image of a 
  bounded domain $\mc{D}\subset\C^n$ by
  the mapping $RC:~(z_1,\dots,z_n)\rightarrow (w_1,\dots,w_n)$ was  called
by Bergman   {\it representative domain} of $\mc{D}$, but we shall  use
  the same denomination  also for homogeneous K\"ahler manifolds. The mapping $RC$ is generally
  holomorphic and one-to-one only  locally. The name {\it
    representative coordinates} was firstly used by Fuks \cite{FUKs1}.

We also  denote by $J(z_1,\dots,z_n)$  the matrix which determines the
change of coordinates $RC: (z_1,\dots,z_n)\rightarrow(w_1,\dots,w_n)$ from
local coordinates on $M$ to the  Bergman representative coordinates \eqref{brc}
\begin{equation}\label{numeDET}
J(z_1,\dots,z_n):=\frac{\pa (w_1, \dots,w_n)}{\pa (z_1,\dots,z_n) }.
\end{equation}
We have already underlined  that 
\begin{Remark}\label{RRE}
\begin{equation}\label{unuB}
J(z_1,\dots,z_n)\bigr|_{z=z_0}=\un,
\end{equation}
and in a neighborhood of $z_0$ the new coordinates $(w_1,\dots,w_n)$ should yield a
basis of local vector fields, i.e.
\begin{equation}\label{notZ}
\det (J(z_1,\dots,z_n))\not= 0.
\end{equation}
\end{Remark} 
In his paper \cite{lu66}, Lu Qi-Keng presented many examples of bounded
domains $\mc{D}\subset\C^n$ in which $\mc{K}_{\mc{D}}(z,\bar{w})\not= 0,
\forall~ z,w\in \mc{D}$,
see the standard definition in \S \ref{bsmb} in the Appendix. The
domains which have a zero free
 Bergman kernel are
called {\it domains satisfying the Lu Qi-Keng conjecture} or {\it Lu Qi-Keng
domains}. It was conjectured in \cite{lu66} that {\it any simply connected
domain  in $\C^n$ is a Lu Qi-Keng domain}.  Swarczy\'nski provided an
example of  an unbounded Reinhardt domain  $ \mc{D}\subset \C^2$  for
which $\mc{K}_{\mc{D}}(z,\bar{w})$ has a zero \cite{mskw1}. Later,  a bounded,
strongly pseudo-convex counterexample to Lu-Qi Keng conjecture was given by Boas \cite{bo}.

It is  known   that \cite{xu}:
\begin{Remark}\label{rempeste2}
The homogenous bounded domains are Lu Qi-Keng
  domains and consequently  the associated Bergman representative
coordinates are globally defined. 
\end{Remark}

In this paper we shall use the name {\it Lu Qi-Keng manifold} to
denote a manifold which has properties similar to the Lu Qi-Keng
domain, i.e. a homogenous K\"ahler manifold for
which \begin{equation}\label{noTT}
K_M(z,\bar{w}) :=(e_z,e_{\bar{w}})\not=
  0, ~\forall~ z,w\in M. \end{equation}
Another name for a manifold $M$ whose points $z\in M$ verify the condition
$K_M(z)\not=0$ would be {\it normal manifold}, a
denomination used by Lichnerowitz  for manifold for which
$\mc{K}_M(z,\bar{z})\not= 0$ or {\it Kobayashi
  manifold}, see \S \ref{KEM} in Appendix.

We recall that in \cite{ber97},  in the context of Perelomov's
CS,  we have called the set
$\Sigma_z:=\{w\in M|K_M(z,\bar{w})=0\}$ the {\it{polar divisor of}} $z\in M$.
We have shown that for {\it {compact homogeneous manifolds for which the
exponential  map from the Lie algebra to the Lie group equals the geodesic
exponential, and in particular for symmetric spaces, the  set $\Sigma_z$ has
the geometric significance of the  cut locus}} $\mb{CL}_z$  (see
p. 100 in \cite{koba2} for the definition of cut locus) attached to the point
$z\in M$ \cite{ber97}. Moreover,  it was  proved  in \cite{SBS} in the context of Rawnsley-Cahen-Gutt
approach to Berezin's CS 
that
if $\Sigma_z=\mb{CL}_z$, then indeed, $\Sigma_z$ is a {\it{polar
    divisor}} in the meaning  of algebraic geometry \cite{gh}. The case of
conjugate locus in $\db{CP}^{\infty}$ as polar divisor in \cite{ber97}
is considered in \cite{LU12}, but it should be taken into
consideration that in such a case the cut locus is identical with first
conjugate point \cite{cr}. 

\subsection{The simplest example: the  Siegel disk  }\label{ssimp}
In order to have a feeling of what the Bergman representative coordinates
are, we take the simplest example of the Siegel disk $\mc{D}_1$ with
the  Bergman kernel function \eqref{ker2}.

\begin{Remark}\label{RRR11} The Bergman  kernel function \eqref{ker2}
  is positive definite  and the
Siegel disk $\mc{D}_1$ is a Lu Qi-Keng domain.  The Bergman representative
coordinates on the Siegel disk are global and the $RC$-transformation is a
K\"ahler homogenous diffeomorphism. 
\end{Remark}
\begin{proof}
Indeed,
the metric matrix $h(w)= \{h_{w\bar{w}}\}$ on $\mc{D}_1$ calculated in
\eqref{mdisk} is
\begin{equation}\label{hmatrix}
h(w)=\frac{\pa^2}{\pa w \pa \bar{w}}\ln K_k(w,\bar{w})=
\frac{2k}{P^2},~
P=1-w\bar{w}.
\end{equation}
We  introduce in the formula \eqref{brc}  the data \eqref{ker2}
and  \eqref{hmatrix} and we get for the Bergman
representative coordinate on the Siegel disk $\mc{D}_1$ 
the expression
\begin{equation}\label{wz}
w_1(w)=P_0\frac{w-w_0}{P_1} , ~ P_0= 1-w_0\bar{w}_0,~P_1=1-\bar{w}_0w.
\end{equation}
The Bergman representative coordinate for the Siegel disk has the
right properties, i.e.  $w_1(w_0)=0$, and \eqref{unuB} and \eqref{notZ}
are also  verified 
because
\begin{equation}
\frac{\pa w_1}{\pa w}\bigr|_{w_1\in\mc{D}_1} =\frac{P_0}{P_1^2}\not= 0.
\end{equation}
Note that the inverse of \eqref{wz} is
\begin{equation}\label{inverswz}
w=\frac{w'_1+w_0}{1+\bar{w}_0w'_1}, \quad w'_1=\frac{w_1}{P_0},
\end{equation}
and
\begin{equation}
K_k(w_1',\bar{w}_1')=\left(\frac{1-|w'_1|^2}{|1+\bar{w}_0w'_1|^2}\right)^{-2k}.
\end{equation}
Also we obtain the expression \begin{equation}
-i \omega_k(w_1')= 2k\frac{\dd w'_1\wedge \dd
\bar{w}'_1} {(1-w'_1\bar{w}'_1)^2}.
\end{equation}
\end{proof}
\subsection{The Siegel-Jacobi disk}\label{SJDC}  
We consider the Siegel-Jacobi disk endowed with the  kernel
function \eqref{ker3}. We proof that:
\begin{Proposition}\label{PR1} The  kernel function
  of the Siegel-Jacobi disk  \eqref{ker3} is positive definite,  
$  \mc{D}^J_1$ is a normal K\"ahler homogeneous Lu Qi-Keng manifold.    The Bergman representative coordinates
    \eqref{brc} are globally defined on $\mc{D}^J_1$ .
\end{Proposition}
\begin{proof}
Using \eqref{hinv}, we particularize  \eqref{brc} to $\mc{D}^J_1$:
\begin{eqnarray*}
w_1(\varsigma) & =  & \left( h^{\bar{1}1}(\varsigma)\frac{\pa}{\pa \bar{z}'}+
  h^{\bar{2}1}(\varsigma)\frac{\pa}{\pa \bar{w}'} \right) X, \\
w_2(\varsigma) & = &  \left( h^{\bar{1}2}(\varsigma)\frac{\pa}{\pa \bar{z}'}+
  h^{\bar{2}2}(\varsigma)\frac{\pa}{\pa \bar{w}'}\right) X, 
\end{eqnarray*}
where, with  the notation of \eqref{ker3} and  \eqref{hot}, we have
\begin{equation}
X:=X(\varsigma,\varsigma')=\ln \frac{K_{k\mu}(\varsigma,\bar{\varsigma}')}{K_{k\mu}(\varsigma')}.
\end{equation}

We find explicitly the representative Bergman coordinates \eqref{brc}
$(w_1,w_2)=RC(z,w)$ on the
Siegel-Jacobi disk $\mc{D}^J_1$
\begin{equation}\label{noi}
\begin{split}
w_1(\varsigma) & = -\bar{\eta}_0w_2+P_0(\eta_1'-\eta_0), \\
w_2(\varsigma) & = \frac{P_0}{P_1}(w-w_0)+\lambda\left[ P_0(\eta_1'-\eta_0)\right]^2, 
\end{split}
\end{equation}
where
\begin{equation}\label{noi1}
\eta_0 : =\eta(z_0,w_0)=\frac{z_0+\bar{z}_0w_0}{P_0}, ~~
\eta_1'=\frac{z+\bar{z}_0w}{P_1} ,~~ \lambda= \frac{\mu}{4k}.
\end{equation}

We have $$(w_1(\varsigma_0),w_2(\varsigma_0))  = (0,0).$$ Note that the Bergman representative coordinates $(w_1, w_2)$
\eqref{noi}, \eqref{noi1} verify the condition \eqref{numeDET}, i.e.
\begin{equation}
J(\varsigma)\bigr|_{\varsigma=\varsigma_0}=\mathbb{1}_2.
\end{equation}
We have also 
\begin{equation}\label{det1}\det J(\varsigma)
\bigr|_{\varsigma\in\mc{D}^J_1}=\left(\frac{P_0}{P_1}\right)^3\not=
0. \end{equation}

Now we reverse  equation \eqref{noi} and we express $\varsigma=(z,w)\in\C\times\mc{D}_1$
of a point in $\mc{D}^J_1$ as function of  the Bergman representative 
coordinates $(w_1,w_2)$ and we get: 
\begin{subequations}\label{noiinvers}
\begin{align}
w & =\frac{y+w_0}{Q},\\
z & =
\frac{z_0-\bar{z}_0y+P_0x}{Q},
\end{align}
\end{subequations}
where
\begin{equation} \label{toatev}
Q = 1+\bar{w}_0y, \quad y =w_2'-\lambda P_0x^2 ,\quad
x=w_1'+\bar{\eta}_0w_2';\quad w_i'=\frac{w_i}{P_0},\quad i=1,2 .
\end{equation}
If we put in \eqref{noiinvers} $(w_1,w_2)=(0,0)$, then
$(w,z)=(w_0,z_0)$.
\end{proof}

Now we express the K\"ahler two-form  \eqref{aab} on the Siegel disk
$\mc{D}^J_1$ in the
Bergman representative coordinates $(w_1,w_2)$ \eqref{noi}. 
The  reverse to  the transformation  $(x,y)\rightarrow (z,w)$ given by
equations \eqref{noiinvers} is
\begin{equation}\label{cool}
x=\frac{z-\eta_0 +\bar{\eta}_0w}{P_1}, \quad y= \frac{w-w_0}{P_1} .
\end{equation} 

We shall  prove that
\begin{Proposition}\label{REMFF} The $RC$-transformation  \eqref{cool}
  $(w,z)\rightarrow(x,y)$ is a
  biholomorphic  mapping  of the Siegel-Jacobi disk to itself. The K\"ahler two-form  $\omega_{k\mu}(z,w) $ of the Siegel-Jacobi disk
  expressed in the variables $(x,y)$  from
  \eqref{noiinvers} has a
value similar to \eqref{aab}:
\begin{equation}\label{newK}
-i\omega_{k\mu}(x,y)=2k\frac{\dd y\wedge \dd
\bar{y}}{(1-|y|^2)^2}+\mu\frac{\mc{B}\wedge\bar{\mc{B}}}{1-|y|^2},
\end{equation}
where the one-form $\mc{B}$ is defined in \eqref{al4}, \eqref{al5}. The action of an element $(l,u)\in {\emph{\text{SU}}}(1,1)\times\C$ on
$(x,y)\in\mc{D}^J_1$,
$(l,u)\circ (x,y)=(x',y')$, where
\begin{equation}\label{srsr}
l=\left(\begin{array}{cc} r & s\\
\bar{s} & \bar{r}\end{array}\right),\quad |r|^2-|s|^2=1, 
\end{equation}
is similar to the action  \eqref{dg}, \eqref{xxx}, i.e.
\begin{equation}\label{xxx1}
x'=\frac{x+u-\bar{u}y}{\bar{s}y+\bar{r}},\quad
y'=\frac{ry+s}{\bar{s}y+\bar{r}},
\end{equation}
where the parameters of the matrix $l\in\emph{\text{SU}}(1,1)$ and $u\in\C$
are expressed as function of the matrix  $g \in {\emph{\text{SU}}}(1,1)$ \eqref{dg} and $\alpha\in\C$
describing the action $(g,\alpha)\circ(w,z)=(w_1,z_1)$ in \eqref{dg},
\eqref{xxx}:
\begin{subequations}
\begin{align}
r & =\frac{a-\bar{b}w_0+\bar{w}_0(b-w_0\bar{a})}{P_0}, \\
s & = \frac{b+w_0(a-\bar{a}-{w}_0\bar{b})}{P_0},\\
u &
=\frac{z_0+\alpha-w_0\bar{\alpha}-{\eta}_0(\bar{a}+w_0\bar{b})+\bar{\eta}_0(b+w_0a)}{P_0} .
\end{align}
\end{subequations}
The K\"ahler two-form \eqref{newK} $\omega_{k\mu}(x,y)$ is K\"ahler
homogenous, i.e. it is invariant to the action \eqref{xxx1}. The 
$RC$-transformation  \eqref{cool} for the Siegel-Jacobi disk is a
homogenous  K\"ahler
diffeomorphism and the
representative manifold of the Siegel-Jacobi disk is the Siegel-Jacobi
disk itself.

We express the resolution of unit  on the Siegel-Jacobi disk \eqref{ofi} in the variables
$(x,y)$ related to the variables $(z,w)$ by \eqref{noiinvers}. The
measure \eqref{ofi3}, $G^J_1$-invariant to the action \eqref{xxx1}, has the expression:
\begin{equation}\label{newm}
d \nu(x,y) =\mu\frac{\dd \Re y\dd\Im y}{(1-y\bar{y})^3}\dd \Re x \dd \Im x.
\end{equation}

We express the reproducing kernel \eqref{hot} in the variables
$(x,y)$ given by \eqref{toatev}. For $P$ we have the expression
\eqref{al1}, while for  $F(x,y)$ we have the expression:
\begin{align*}
2(1-|y|^2)|Q|^2 F(x,y)   & = |y|^4(w_0\bar{\eta}_0\bar{z}_0 + cc)+
|y|^2(yx_1+cc)+P_0(\bar{w}_0\bar{x}^2y^2+cc)\\
 & 
 +y[-x_1+P_0\bar{x}(2\bar{w}_0x+\bar{x}(1+|w_0|^2)] +cc\\
&  + |y|^2[|z_0|^2-\eta_0(\bar{z}_0+P_o\bar{x})-cc] + \eta_0+x(1+w_0) + cc, \\
& {\emph{\text{where}}} \quad x_1=\bar{z}_0^2-2\bar{w}_0\eta_0P_0\bar{x}. 
\end{align*}
\end{Proposition}
\begin{proof}

 In order to do this calculation,
we use the values of $(z,w)$ from \eqref{noiinvers} and we have
\begin{subequations}\label{toatealea}
\begin{align}
P & =1-w\bar{w}=P_0\frac{1-|y|^2}{|Q|^2},\label{al1} \\
\dd w & = P_0\frac{\dd y}{Q^2}, \label{al2}\\
\dd z & = P_0\frac{-(\bar{\eta}_0+\bar{w}_0x)\dd y+Q\dd x
  }{Q^2}, \label{al3}\\
\eta(z,w) & = \eta_0 +\frac{B +w_0\bar{B}}{1-|y|^2},~
B:= x+\bar{x}y, \label{al4}\\
\mc{A} & = P_0\frac{\mc{B}}{Q}, ~ \mc{B}=\dd x+ \bar{\eta}(x,y)\dd y,
~\eta(x,y)=\frac{B}{1-|y|^2}.
\label{al5}
\end{align}
\end{subequations}
  From \eqref{al1},  we find that $w\in\mc{D}_1\Longleftrightarrow
y\in\mc{D}_1$, i.e.
$$1-|w|^2> 0\quad\Longleftrightarrow\quad 1-|y|^2 >0.$$

Note that the change of coordinates $(z,w)\rightarrow
(w_1,w_2)\rightarrow
(y,x)$ is well defined on $\mc{D}^J_1$ because 
$$\frac{\pa(y,x)}{\pa (z,w)}=\frac{\pa(y,x)}{\pa (w_1,w_2)}\frac{\pa(w_1,w_2)}{\pa (z,w)},$$
and
$$\det \frac{\pa (y,x)}{\pa (w_1,w_2)}=\frac{1}{P_0},$$
so,   with \eqref{det1},  we have
 $$\det \frac{\pa (y,x)}{\pa
   (z,w)}\bigr|_{\mc{D}^J_1}=\frac{P_0^2}{P_1^3} \not= 0.$$
Taking
into account \eqref{al1}, \eqref{al2} and \eqref{al3}, we get for the
measure \eqref{ofi3} the expression \eqref{newm}. 
\end{proof}

\section{Appendix}

\subsection{Bergman pseudometric and  Bergman metric}\label{bsmb}

Let $M$ be a complex $n$ -dimensional  manifold.
We  consider three Hilbert spaces, denoted \fl, $\gf_2(M)$ and
$\gf_n(M)$, and, under some conditions, we shall establish 
correspondences  between them. 

Let $M=G/H$ be a homogenous K\"ahler manifold. We consider  the Hilbert
space \fl~   with scalar product \eqref{scf}, a particular realization
of weighted Hilbert space $\Hi_f$ \eqref{HIF}  corresponding to a
constant 
$\epsilon$-function  \eqref{EPSF}, endowed with the base of
functions $\{\varphi_i\}_{i=0,1,\dots,}$ verifying \eqref{functphi}. 
To the K\"ahler potential $f=\ln(K_M)$ we
associate the balanced  metric \eqref{herm}.

For a $n$-dimensional complex manifold $M$, let us denote by  $\gf_2(M)$  the Hilbert space of square integrable functions with respect
to the scalar product
\begin{equation}\label{sc2f}
(f,g)_{\gf_2(M)}=i^{n^2} \int_M\bar{f}(z)g(z) \dd z\wedge\dd \bar{z}
= \int_M\bar{f}(z)g(z) \dd V,
\end{equation}
where we have introduced the abbreviated notation $\dd z= \dd z_1\wedge\dots\wedge \dd
z_n$ and $\dd V$ has the expression  given in \eqref{OMMF}. This space
was considered by Stefan Bergman for bounded domains
$\mc{D}\subset\C^n$ \cite{berg,berg1,berg2}.

Let $\{ \Phi\}_{i=0,1,\dots}$ be a orthonormal basis of
$\gf_2(M)$. Then to  the {\it Bergman kernel function} 
\begin{equation}\label{BERGF}
\mc{B}_2(z,\bar{w})=\sum_{i=0}^{\infty}\Phi_i(z)\bar{\Phi}_i(w),
\end{equation}
 {\it  it is associated the Bergman
  metric}
\begin{equation}\label{berg11}
\dd s^2_{\mc{B}_2} (z)=\sum_{i,j=1}^n\frac{\pa^2 \ln \mc{B}_2(z,\bar{z})}{\pa
    z_i\pa\bar{z}_j}\dd z_i \otimes\dd \bar{z}_j.
\end{equation}

Let  now denote by $\gf_n(M)$   the set of square integrable holomorphic $n-$forms
on the complex $n$-dimensional manifold $M$, i.e.  
\begin{equation}\label{FNN} i^{3n^2}  \int_M \bar{f}\wedge f <\infty
  . \end{equation}
This space was considered by  Weil \cite{AWeil}, Kobayashi
\cite{koba,koba70},  Lichnerowicz \cite{Lich}, while for homogeneous
manifolds, see Koszul \cite{JLK} and Piatetski-Shapiro
\cite{ps1,ps2}.  
{\it The vector space   $\gf_n(M)$  is a separable complex Hilbert space with
countable base}  (cf  Corollaries at p. 60 in \cite{AWeil})     and if we
write $ f=f^*\dd z$, $g=g^*\dd z$, then the Hilbert space 
has  the  inner product
given by
\begin{equation}\label{sc3}(f,g)_{\gf_n(M)}=i^{3n^2} \int_M \bar{f}
  \wedge g =  \int_M \bar{f^*}g^*\dd V , \end{equation}
which {\it is invariant  to  the action of holomorphic transformations of}
$M$ (cf the Theorem at p. 353 in \cite{Lich}).

Let $f_0(z), f_1(z), \dots$ be a complete orthonormal basis of the
space   $\gf_n(M)$  of square integrable holomorphic $n$-forms on
$M$. 
 It can be shown (cf  Theorem 1 at  p. 357 in
\cite{Lich}) that  the series 
$$\sum_{k=0}^{\infty}f_k(z)\wedge \bar{f}_k(w) $$ {\it is absolutely
convergent and defines  a form  $\mc{K}_M(z,\bar{w})$  holomorphic   in $z$ and $\bar{w}$}, 
\begin{equation}\label{berFORM}
\mc{K}_M(z,\bar{w})=\sum_{k=0}^{\infty}f_k(z)\wedge \bar{f}_k(w),\end{equation}
  called the {\it Bergman kernel form
  of the manifold} $M$, {\it which is    independent of the
base  and is invariant under the group of
  holomorphic transformations of $M$} \cite{koba,Lich}. 
 
For $f_i(z)\in\gf_n(M)$, let us denote by  $f^*_i(z)$ the holomorphic
function  defined locally  such that 
 \begin{equation}\label{cele2}f_i(z)=f^*_i(z)\dd z.
\end{equation}
Then we have the relation 
\begin{equation}\label{FForm}\mc{K}_M(z,\bar{w})=\mc{B}_n(z,\bar{w})
  \dd z \wedge \dd \bar{w},\end{equation}
where {\it the Bergman kernel $\mc{B}_n(z,\bar{w}) $  of the complex manifold $M$} (as it is
  called by Lichnerowicz  in \cite{Lich})
has the series expansion \eqref{sumBERG1}
\begin{equation}\label{sumBERG1}
\mc{B}_n(z,\bar{w})=\sum_{i=0}^{\infty} f^*_i(z)\bar{f}^*_i(w). 
\end{equation}

A complex analytic manifold $M$ for which $ \mc{K}_M(z,\bar{z})$ {\it{is
different of zero in every point of}}  $z\in M$ is 
called by Lichnerowicz  {\it  normal manifold}  (cf p. 367 in \cite{Lich}). The
covariant symmetric tensor $t$ of type $(1,1)$,  called {\it the Bergman
tensor of the complex normal manifold $M$}, has locally the components 
\begin{equation}\label{TTT}
t_{\alpha\bar{\beta}}=\frac{\pa^2}{\pa z _{\alpha}\pa
  \bar{z}_{\beta}}\ln (\mc{B}_n(z,\bar{z})), ~\alpha,\beta =1,\dots,n.
\end{equation}
{\it The Bergman tensor of a complex  normal  manifold $M$ is
  invariant under all holomorphic  transformations of the manifold}
$M$ (cf the Theorem at p. 369 in \cite{Lich}).

Assume  that we are in the points $z\in M'\subset M$ where
$\mc{B}_n(z,\bar{z}) >0$ and let us define the Bergman pseudometric (in the
sense of \cite{aisk})
\begin{equation}\label{kobaM}
\dd  s^2_{\mc{B}_n}(z)=\sum_{i,j=1}^nt_{i\bar{j}} \dd z_i \otimes \dd\overline{z}_j.
\end{equation}

Following  \S 5 and \S 8 in \cite{Lich}, for every $z\in M$, let us
introduce the subspace of $\gf_n(M)$
$$\gf'_n(z)=\left\{\gamma\in \gf_n(M)|\gamma(z)=0\right\}.$$
Let $\alpha_0\in\gf_n(M)$ such that
\begin{equation}\label{alph0}
\alpha_0(z)\not= 0,~~(\alpha_0,\alpha_0)_{\gf_n(M)}=1,~~ (\alpha_0,\gamma) _{\gf_n(M)}=0,
~~\forall \gamma\in\gf'_n(z).\end{equation}
{\it An adopted orthonormal basis} of $\gf_n(M)$ is a basis $\alpha_0,
\alpha_1,...$  such that $\alpha_0$ verifies \eqref{alph0}  and
$\alpha_i(z)=0$ for
$i>0$. Then $$\mc{K}_n(z,\bar{z})=\alpha_0(z)\wedge\bar{\alpha}_0(z),~~
\mc{B}_n(z,\bar{z})=|\alpha_0(z)|^2>0
\quad{\text{if~~~}}\gf'_n(z)\not=\gf_n(M),$$
and we get
$$\dd
s^2_{\mc{B}_n}(z)=\frac{1}{\mc{B}_n(z,\bar{z})}\sum_{k=1}^{\infty}|\dd\alpha_k(z)|^2>0, $$
i.e. $\dd s^2_{\mc{B}_n}(z)$ is positive definite in $M'$.

 In fact, we have
the assertions:
{\it The  quadratic form \eqref{kobaM} is positive semi-definite 
   in the points of $M'\subset M$ and invariant under holomorphic
   transformations in $M$
and K\"ahlerian. For Kobayashi manifolds the Bergman metric
\eqref{kobaM} is positive definite.  Bounded
domains in $\C^n$ have positive  definite Bergman metric},  see  
Theorem 3.1 in \cite{koba},  Ch III
in  \cite{AWeil}  and \S 3.2 in  \cite{tri}. 

Now we establish correspondences  of the basis of the  Hilbert spaces denoted \fl, $\gf_2(M)$, and
$\gf_n(M)$, comparing respectively the formulae \eqref{scf},
\eqref{sc2f} and \eqref{sc3}, in the situation that this is
possible.
\begin{Proposition}\label{BIGCOR}
Let $M$ be a $n$-dimensional complex manifold $M$.  We have  the identity
\begin{equation}
\mc{B}_n(z,\bar{w})=\mc{B}(z,\bar{w}), {\emph{\text{~where~~}}}  \mc{B}(z,\bar{w}):=\mc{B}_2(z,\bar{w}).
\end{equation}

The Bergman metrics \eqref{berg11} and \eqref{kobaM} coincides
\begin{equation}\label{TRUB}
\dd s_{\mc{B}_2}^2 = \dd s_{\mc{B}_n}^2  = \dd
s_{\mc{B}}^2, \text{{\emph{~where~}}}\dd
s_{\mc{B}}^2= \sum_{i,j=1}^n\frac{\pa^2\mc{B}(z,\bar{z})}{\pa z_i\pa
  \bar{z}_j}\dd z_i \otimes \dd \bar{z}_j,
\end{equation}
but they are different from the metric $\dd s^2_M(z)$ given by \eqref{herm}.

Let now $M=G/H$ be a  homogeneous K\"ahler manifold. We have also the correspondence of the base functions \eqref{cele2}
$f_i\leftrightarrow f^*_i$ defining $\gf_n(M)$, $\Phi_i$ defining
$\gf_2(M)$, and $\varphi_i$ defining \fl
\begin{equation}\label{cor} 
f^*_i\leftrightarrow   \Phi_i \leftrightarrow   
 ~\sqrt{\Upsilon} \varphi_i,~~i=0,1, \dots,  ~~~~ ~~~~~\Upsilon=\frac{\mc{G}(z)}{K_M(z)}.
\end{equation}

   On homogeneous manifolds,   
the Bergman $2n$-form \eqref{FForm} can be expressed in function
     of the normalized Bergman kernel \eqref{kmic}  by the
     formula  
\begin{equation}\label{BQQ}
\mc{B}_n(z,\bar{w})={\sqrt{\Upsilon(z){\Upsilon}(w)}} K_M(z,\bar{w})= \sqrt{\mc{G}(z){\mc{G}}(w)}\kappa_M(z,\bar{w}).\end{equation}  
In particular, we have the relations:
\begin{equation}\label{BQQPART}
 \mc{B}_n(z,\bar{z})\equiv 
\mc{G}(z).
\end{equation} 
If $\omega^1_{M}$ is the K\"ahler two-form  associated to the
K\"ahler potential    $\mc{B}(z,\bar{z})$, then  
\begin{equation}\label{altom}
 \omega^1_{M}= i\pa\bar{\pa}\ln(\mc{G})=-\rho_M,\end{equation} 
where $\rho_M$ denotes the Ricci form \eqref{RICCI}
and the homogenous manifold $M$ is Einstein with respect to the Bergman metric
$\dd s^2_{\mc{B}}$.                         
\end{Proposition}
\subsection{Embeddings} \label{EMBB}

Let $M$ ( $N$)  be a complex manifold of dimension $m$ (respectively,
$n$). The continuous mapping $f:~M~\rightarrow ~N$ is called
{\it{holomorphic}} if the coordinates of the image point are expressed
as holomorphic functions of  those of the original point. $f$ is an
{\it {immersion}} if $m\le n$ and if the functional matrix is of rank
$m$ in all points of $M$. An immersion is an {\it{embedding}}  if $f(x)=f(y)$,
for $x,y\in M$ implies $x=y$, see e.g. p. 60 in \cite{chern}.

It can be proved (see Theorem (E) at p. 60 in \cite{chern}): {\it Let
  $N$ be a K\"ahlerian manifold and let $f: ~M~\rightarrow~N$ be a
  holomorphic immersion. Then $M$ has a K\"ahlerian structure}.

We are  concerned with  manifolds $M$ which admits an
embedding in some projective Hilbert space attached to a holomorphic
line bundle $\gl \rightarrow  M$
\begin{equation}\label{emb}
\iota :
 M   \hookrightarrow \gpl .
\end{equation}
\subsubsection{The compact case}Firstly, we recall  the case of compact manifolds $M$.

1). A holomorphic line bundle $\gl$ on a compact complex manifold  $M$
is said {\it very ample} \cite{ss} if:

C$_1$) the set of divisors is without base points, i.e. there exists
a finite set of global sections $s_1,\ldots ,s_N\in \Gamma (M, \gl)$
such that for each $m\in  M$ at least one $s_j(m)$ is not zero;

C$_2$) the holomorphic map $\iota_{\gl}:M\hookrightarrow\gcp^{N-1}$
given by
\begin{equation}
  \iota_{\gl}=[s_1(m):\ldots :s_N(m)]\label{scufund}
\end{equation}
is a holomorphic embedding.

 So,
 $\iota_{\gl}: M \hookrightarrow\gcp^{N-1}$ is an embedding  (cf \cite{gh})
 if the following conditions are fulfilled:

\AA$_1)$ the set of divisors is without base points;

\AA$_2)$ the differential of the map $\iota$ is nowhere degenerate;

\AA$_3)$ the map $\iota$ is one-one, i.e. for any $m, m' \in M$
there exists $s\in \Gamma (M,\gl)=H^0(M ,\mc{O}(\gl))$
  such that $s(m)=0$ and $s(m')\not= 0$, cf  Proposition 4 \S~22 p. 215 in Ref. \cite{serre}.

\subsubsection{Kobayashi embedding}\label{KEM}
Now we discuss the construction of the embedding \eqref{emb}
  for noncompact complex  manifolds $M$. Then the projective
Hilbert space  in \eqref{emb} is infinite dimensional \cite{koba}.

Following Kobayashi
\cite{koba}, we recall the conditions in which the Bergman
pseudometric \eqref{kobaM} is a metric on complex manifolds $M$.

The analogous of conditions \AA$_1$)-\AA$_3$) used by Kobayashi in the
 noncompact  case are:

A$_1$) for any $z\in M$, there exists a square integrable $n-$form
$\alpha$ such that $\alpha(z)\not= 0$, i.e. the Kernel form
$\mc{K}_M(z,\bar{z})$ is different from zero in any point of $M$; 

A$_2$) for every holomorphic vector $Z$ at $z$ there exists a square integrable
$n$-form $f$ such that $f(z)=0$ and $Z(f^*)\not= 0$, where $f=f^*\dd
z_1\wedge\dots\wedge \dd z_n$;

A$_3$) if $z$ and $z'$ are two distinct points of  $M$, then there is a
 $n-$form $f$ such that $f(z)=0$ and $f(z')\not= 0$. The bounded
 domains in $\C^n$ have also more  specific properties, as
 $C$-hyperbolicity and hyperbolicity in the sense of Kobayashi \cite{tri}.

Piatetski-Shapiro calls a manifold with properties A$_1$)-A$_2$) a
{\it Kobayashi manifold} \cite{ps1}. The main property of the Kobayashi manifolds
similar with that of bounded domains in $\C^n$ is the existence of
a      {\it  positive definite}  Bergman metric, which we present below.

Let $\alpha$ be the holomorphic $n$-form defined at A$_1$). Then to any $f\in\gf_n(M)$ we put
into correspondence an analytic function $h(z)$ defined by the
relation
$f=h\alpha$.  The set of such functions generates a Hilbert space of
functions \gf~ with scalar product 
\begin{equation}\label{2h}
(h,h')=i^{3n^2} \int_M \bar{h}
h' \bar{\alpha}\wedge \alpha , 
\end{equation}
 isomorphic with $\gf_n(M)$. We see below that under some
conditions this space \gf~ can be identified with the Hilbert space
\fl~ with the scalar product \eqref{scf}, see Remark \ref{RRR}.

 Then $\gl =\gf ^*$.  Let $z=(z_1,\ldots ,z_n)$ be a local
 coordinate system. Let $\iota '$ be the mapping which sends $z$ into an
 element $\iota '(z)$ of \gl~  defined by the paring $<\iota' (z),f>=
f^*(z)$, where $f(z_1,\ldots ,z_n)=f^*\dd z_1\wedge \cdots \wedge \dd z_n\wedge
\dd \overline{z}_1\wedge \cdots \wedge \dd \overline{z}_n$. Then $\iota '(z)\not=
0$ if a condition analogous to condition \AA$_1$) in the noncompact case
 is satisfied. Then $\iota
=\xi \circ \iota '$ is independent of local coordinates and is continuous and
complex analytic. 

Note that if $M$ satisfies  A$_3$), then it automatically satisfies   A$_1$).

Kobayashi has shown (see theorems 3.1, 8.1, 8.2  in \cite{koba}) that 
\begin{Proposition}\label{Bigth}
Let $M$ be a complex manifold. Condition {\emph{A}}$_1$\emph{)} implies that  
the quadratic form \eqref{kobaM} is positive semi-definite, invariant
under the holomorphic transformations of $M$, $\dd s^2_{\mc{B}_n}$  is induced
from the canonical Fubini-Study K\"ahler metric \eqref{FBST} and we
have a relation similar  to \eqref{KOL}:
\begin{equation}\label{cedDIF}\dd s_{\mc{B}_n}^2=\iota
^*(\dd s^2_{FS}).\end{equation}  The differential $\dd \iota$ is not
singular at any point of $M$ if and only if $M$ satisfies {\emph{A}}$_2$\emph{)}.
If  $M$ satisfies {\emph{A}}$_1$\emph{)} and {\emph{A}}$_2$\emph{)}, then $\iota$  is an
isometric immersion of $M$ 
into $\db{CP}^{\infty}$. If $M$ is a complex manifold
satisfying {\emph{A}}$_2$\emph{)} and {\emph{A}}$_3$\emph{)}, then $\iota$  is an isometrical
imbedding of $M$ into  $\db{CP}^{\infty}$  and we have \eqref{cedDIF}.
\end{Proposition}


\begin{Remark}\label{RRR}
Let us suppose that the homogeneous K\"ahler  manifold $M$ is such that we have the space
$\fl\not= 0$
with positive definite 
 kernel function  \eqref{funck}  and 
$M$ is a Lu Qi-Keng  manifold. The Bergman kernel form \eqref{berFORM} and the
Bergman kernel function are related by \eqref{BQQ}.
\end{Remark}

{\bf In conclusion}, in  this paper we have illustrated elements of Berezin's quantization
on the partially bounded manifold $\mc{D}^J_1$.  Propositions
\ref{prop2} and \ref{REM9} have been already pu\-blished \cite{csg},
 but in the present paper we added   new   assertions and statements. We have considered
 representative Bergman coordinates on  Lu Qi-Keng and normal
 manifolds. The Propositions \ref{PR1},
\ref{REMFF}, \ref{BIGCOR} and the Remark \ref{RRR}  contain  the main
results of this paper. \\[1ex]

{\bf Acknowledgement}  This research  was conducted in  the  framework of the 
ANCS project  program PN 09 37 01 02/2009  and     the UEFISCDI - Romania 
 program PN-II Contract No. 55/05.10.2011. 


\begin{thebibliography}{99}



\bibitem{arr}C. Arezzo and A.  Loi,  {\it  Moment maps, scalar curvature
    and quantization of K\"ahler manifolds}, Commun. Math. Phys. {\bf  246} (2004) 543--549

\bibitem{ball}W. Ballmann, 
{\it Lectures on K\"ahler Manifolds},  ESI Lectures in Mathematics and
Physics, European Mathematical Society (EMS), Z\"urich, 2006

\bibitem{ber97} S. Berceanu, {\it Coherent states and geodesics: cut locus and conjugate
 locus},
 J. Geom.  Phys.   {\bf 21}  (1997)   149--168 

\bibitem{SBS}S. Berceanu and M. Schlichenmaier, {\it Coherent state
    embeddings, polar divisors and Cauchy formulas},  J. Geom.
  Phys. {\bf 34}  (2000) 336--358


\bibitem{last}S. Berceanu and A. Gheorghe,
{\it  Differential operators on orbits of coherent states},
 Romanian J.
Phys. {\bf  48} (2003) 545--556; arXiv: 0211054 [math.DG]



\bibitem{sb6}S. Berceanu,    Realization of coherent state algebras
by differential operators,
in  {\it {Advances in Operator Algebras and Mathematical Physics}},
 Edited  by F. Boca, O. Bratteli, R. Longo and  H. Siedentop, 
The Theta Foundation, Bucharest, 1--24,  2005


\bibitem{jac1} S. Berceanu,  {\it A holomorphic representation of the Jacobi algebra},
Rev. Math. Phys.  {\bf 18}  (2006) 163-199; {\it  Errata},   Rev.  Math. Phys.
{\bf 24}  (2012) 1292001 (2 pages), arXiv: 0408219   [math.DG] 


\bibitem{sbj} S. Berceanu,   A holomorphic 
representation of Jacobi algebra in several 
dimensions, in {\it { Perspectives in Operator Algebra and Mathematical
Physics}}, Edited by F.-P. Boca, R. Purice and  S. Stratila, The Theta Foundation,
 Bucharest 1-25, 2008;  arXiv: 0604381 [math.DG]

\bibitem{sbcg}S. Berceanu  and A. Gheorghe, {\it Applications of the Jacobi group to
Quantum Mechanics},  Romanian J.  Phys. {\bf 53}  (2008)
 1013-1021; arXiv: 0812.0448 [math.DG]

 
\bibitem{ber9} S. Berceanu, The Jacobi group and the  squeezed states
  - some comments, in   {\it  AIP Conference Proceedings Volume}  {\bf 1191}, {\it
    Geometric methods in Physics},  pp 21--29,
   2009, Edited by  P. Kielanowski, S. T. Ali,  A.
  Odzijewicz, M. Schlichenmaier and  Th.  Voronov; arXiv:0910.5563v1 [math.DG] 

\bibitem{gem}S. Berceanu and  A. Gheorghe, {\it On the geometry of Siegel-Jacobi domains},
 Int. J. Geom. Methods Mod. Phys. {\bf 8}  (2011) 1783--1798;
 arXiv: 1011.3317 [math.DG] 


\bibitem{FC1} S. Berceanu,    Classical and
  quantum evolution   on the Siegel-Jacobi manifolds,   
in {\it  Proceedings of the XXX  Workshop on Geometric Methods
  in Physics},  Edited by P. Kielanowski,  S. T.  Ali, A. Odzijewicz,
M. Schlichenmaier and  T. Voronov,  Trends in Mathematics, Springer Basel AG, 43--52, 2012 

\bibitem{nou}S. Berceanu,  {\it A convenient coordinatization of Siegel-Jacobi
    domains},   Rev.  Math. Phys.  {\bf 24}  (2012) 1250024 (38 pages);
  arXiv: 1204.5610 [math.DG] 


\bibitem{FC} S. Berceanu,  {\it Consequences of the fundamental conjecture for the motion
  on the  Siegel-Jacobi disk},   Int.  J.  Geom. Methods Mod. Phys. {\bf
  10} (2013) 1250076 (18 pages);   arXiv: 1110.5469v2 [math.DG] 


\bibitem{csg}S. Berceanu, {\it Coherent states  and geometry on the
    Siegel-Jacobi disk}, Int.  J. Geom. Methods Mod. Phys. (2014) 1450035 (25 pages);
arXiv: 1307.4219v2 [math.DG]


\bibitem{ber73}F. A. Berezin, {\it Quantization in complex
bounded domains} (Russian),  Dokladi Akad. Nauk SSSR, Ser. Mat. {\bf
211} (1973) 1263--1266

\bibitem{ber74}F. A. Berezin, {\it Quantization}  (Russian),
Izv. Akad. Nauk SSSR Ser. Mat. {\bf 38}  (1974) 1116--1175

\bibitem{ber75}F. A. Berezin, {\it Quantization in complex symmetric
  spaces}  (Russian), Izv. Akad. Nauk SSSR Ser. Mat. {\bf 39} (1975) 363--402

\bibitem{berezin}F. A. Berezin, {\it The general concept of quantization},
  Commun. Math. Phys. {\bf 40} (1975) 153--174


\bibitem{berg}S. Bergman, {\it \"Uber die Existenz von
    Repräsentantenbereichen in der Theorie der Abbildung durch Paare
    von Funktionen zweier komplexen Ver\"anderlichen}, Math. Ann. {\bf
    102} 430-446 (1930)

\bibitem{berg1} S. Bergman, {\it \"Uber die Kernfunktion eines Bereiches und ihr Verhalten am Rande},
J. Reine Angew. Math.  {\bf  169}  (1932)  1--42;  {\bf 172}  (1934)
 89-128
\bibitem{berg2} S. Bergman, {\it  Sur les fonctions orthogonales de plusieurs variables complexes}, Mem. Sci. Math.
Paris  {\bf 106}  (1947)
\bibitem{berg3}  S. Bergman, {\it Sur la fonction-noyau d'un domaine}, 
   Mem. Sci. Math.
Paris  {\bf  108}   (1948)


\bibitem{bs} R.  Berndt   and  R.  Schmidt,     {\it Elements of the representation
theory of the Jacobi group}, Progress in Mathematics  {\bf 163},
Basel, Birkh\"auser, 1998

\bibitem{bo}H.  P.  Boas, {\it Counterexample to the Lu Qi-Keng
    conjecture}, Proc. Amer. Math. Soc. {\bf 97}
(1986) 374--375

\bibitem{calabi}E. Calabi, {\it Isometric imbedding of complex
    manifolds}, Ann. Math. {\bf 58} (1953) 1--23


\bibitem{Cah}M. Cahen, S. Gutt and  J. Rawnsley, {\it Quantization of
    K\"ahler manifolds I: geometric interpretation of Berezin's
    quantization}, J. Geom. Phys. {\bf 7} (1990) 45--62

\bibitem{cah}M. Cahen, S. Gutt and J. Rawnsley, {\it Quantization of K\"ahler
  manifolds. II}, Trans. Math. Soc. {\bf 337}
  (1993) 73--98



\bibitem{Cay}A. Cayley, {\it A six memoir upon quantics},
  Phyl. Trans. R. Soc. London {\bf 149} (1859),  61--90;  561--592  in {\it Collected
    mathematical papers}, Vol II, Cambridge University
   Press, Cambridge, 1989


\bibitem{chern}S. S. Chern, {\it Complex manifolds without potential
    theory}, Spriger-Verlag, Berlin, 1979


\bibitem{cr} R. Crittenden, {\it Minimum and conjugate points in
    symmetric spaces}, Canad. J. Math. {\bf 14}  (1962) 320--328 

\bibitem{davi} J. Davidov, {\it The representative domain of a complex
    manifold and the Lu Qi-Keng conjecture},
 C. R.  Acad. Bulgare Sci. {\bf 30}
(1977) 13--16

\bibitem{dinew}Z. Dinew, {\it On the Bergman  representative
    coordinates}, Sci. China Math. {\bf 54} (2011) 1357--1374

\bibitem{don} S. Donaldson, {\it  Scalar curvature and projective
    embeddings, I},  J. Diff. Geom. {\bf 59}   (2001) 479--522 

\bibitem{ez}M.   Eichler  and  D.  Zagier,   \emph{The theory of Jacobi forms},
Progress in Mathematics {\bf 55},  Boston,
Birkh\"auser,  1985 

\bibitem{FUKs1} B. A. Fuks,  {\cyr{ Teoriya analitičeskih funkciĭ mnogih
  kompleksnyh peremennyh}}, (Russian) [Theory of Analytic Functions of
  Several Complex Variables] OGIZ, Moscow-Leningrad, 1948
\bibitem{FUKs2} B. A. Fuks,
 {\cyr {Spetsialʹnye glavy teorii
      analiticheskikh funktsiĭ
mnogikh kompleksnykh peremennykh.}} (Russian) [Special chapters in the
theory of analytic functions of several complex variables]
Gosudarstv. Izdat. Fiz.-Mat. Lit., Moscow 1963 
  

\bibitem{gpv} S. G. Gindikin, I. I. Pjateccki\v{i}-\v{S}apiro and  E. B. Vinberg, 
    Homogenous K\"ahler manifolds, in {\it Geometry of homogenous
      bounded domains}, Edited by E. Vesentini, Springer-Verlag, Berlin, 
    Heidelberg, 2011, Lectures given at the Summer School of the
    C.I.M.E. held at Urbino, Italy, July 3--13, 1967


\bibitem{tri}R. E. Green, K.-T. Kim and  S. G. Krantz,  {\it  The geometry of
  complex  domains}, Progress in Mathematics {\bf 291}, Birkh\"auser, 
  Boston,  2011


\bibitem{gh}P. Griffith and J. Harris, {\it Principles of Algebraic Geometry},
Wiley, New York, 1978

\bibitem{hua} {L. K. Hua},
\emph{Harmonic Analysis of Functions of Several Complex Variables
in the Classical Domains}, \rm{Translated from Russian by Leo
Ebner and Adam Kor\'anyi, Amer. Math. Soc., Providence, R.I.,
1963}

\bibitem{aisk}A. V. Isaev and S. G. Krantz, {\it Invariant Distances
    and Metrics in Complex Analysis},  Notices Amer.  Math.  Soc.  {\bf 5}
  (2000) 546--553

\bibitem{koba}S. Kobayashi, {\it Geometry of bounded domains},
  Trans. Amer. Math. Soc {\bf 92} (1959) 267--290




\bibitem{koba1}S. Kobayashi and K. Nomizu, {\it Foundations of
    differential geometry}, Vol I, Interscience publishers, New York, 1963

\bibitem{koba2}S. Kobayashi and K. Nomizu, {\it Foundations of
    differential geometry}, Vol II, Interscience publishers, New York, 1969

\bibitem{koba70}
S.  Kobayashi, {\it Hyperbolic manifolds and holomorphic mappings. An
  introduction}. Second edition. World Scientific Publishing
Co. Pte. Ltd., Hackensack, NJ, 2005  

\bibitem{Kos}B. Kostant,  Quantization and unitary
    representations I. Prequantization, in {\it Lecture Notes in Mathematics} {\bf 170},  edited
  by C. T. Tam, Springer-Verlag, Berlin, 87--208, 1970


\bibitem{JLK}J. L. Koszul, \textit{Sur la forme hermitienne canonique
    des espaces homogenes complexes}, Canad. J. Math. {\bf 7} (1955)
  562--576


\bibitem{Lich}A. Lichnerowicz, \text{\it Vari\'et\'es complexes et
    tenseur de Bergman},  Ann. Inst. Fourier, Grenoble {\bf 15} (1965) 345--408


\bibitem{lis}W. Lisiecki, {\it Coherent state representations. A survey},
  Rep. Math. Phys. {\bf 35} (1995) 327--358

\bibitem{alo}A. Loi and R. Mossa, \textit{Berezin quantization of
    homogenous bounded domains}, Geom. Dedicata {\bf 161} (2012) 119--128

\bibitem{lu66} Q.-K.  Lu, {\it On  K\"ahler manifolds with constant
    curvature}, Acta. Math. Sin. {\bf 66} (1966) 269--281

\bibitem{LU08} Q.-K.  Lu, {\it Holomorphic invariant forms on a
    bounded domain}, Sci. China. Math. {\bf 51A} (2008) 1945--1964 

\bibitem{LU12} Q.-K.  Lu, {\it The conjugate points of  $\db{CP}^{\infty}$ and zeros of the
    Bergman kernel}, Acta Math. Sin. (English Series)  {\bf 28}
 (2012)  295--298


\bibitem{mor}A. Moroianu,  {\it Lectures on K\"{a}hler Geometry},
  Cambridge University Press, London Mathematical Society Student
  Texts {\bf 69}, Cambridge University Press, Cambridge, 2007


\bibitem{neeb}K.-H. Neeb, {\it Holomorphy and Convexity in Lie Theory}, de
 Gruyter Expositions in Mathematics {\bf 28}, Walter de Gruyter,
 Berlin, New
York, 2000


\bibitem{perG}A. M.  Perelomov, {\it Generalized Coherent States and their
Applications}, Springer, Berlin, 1986


\bibitem{ps1}I. I. Pjatecki\v{i}-\v{S}apiro,  
{\cyr {Geometriya klassicheskikh
   oblasteĭ i teoriya avtomorfnykh fuiktsiĭ}} (Russian),  
\emph{Geometry of
    Classical Domains and  Theory of Automorphic Functions},
  (Russian), Sovremenye Problemy
  Mathemtiki. Gosudarstv. Izdat. Fiz. Math. Lit,   Moscow, 
 1961


\bibitem{ps2} {I.I. Pyatetskii-Shapiro,} \emph{Automorphic Functions and
the Geometry of Classical Domains}, translated from the Russian,
Mathematics and its applications  {\bf 8}, \rm{Gordon and Breach
  Science publications,  New
York - London - Paris, 1969} 
 

\bibitem{raw}J. H. Rawnsley, {\it Coherent states and K\"ahler
    manifolds}, Quart. J. Math. Oxford Ser. {\bf 28}  (1977) 403--415



\bibitem{satake}I. Satake, {\it Algebraic structures of symmetric domains},
Publ. Math. Soc. Japan {\bf 14},  Princeton Univ. Press,  1980


\bibitem{serre}J.-P. Serre, {\it Faisceaux alg\'ebriques coher\'ents},
Ann. Math. {\bf 61}  (1955) 197--278


\bibitem{ss}B. Shiffman and A. J. Sommese, {\it Vanishing Theorems on Complex
Manifolds}, Progress in Mathematics {\bf  56},  Birkh\"auser, Boston,
1985

\bibitem{mskw1}M. Skwarczy\'nski, {\it The invariant distance in the
    theory of pseudoconformal transformations and the Lu Qi-Keng
    conjecture},  Proc. Amer. Math. Soc. {\bf 22}  (1969) 305--310

\bibitem{mskw}M. Skwarczy\'nski, {\it Biholomorphic invariants related
    to the Bergman function},  Dissertationes Math. (Rozprawy Mat.)
  173 (1980), 59 pp  




\bibitem{xu}Y. Xu,  {\it Theory of Complex Homogeneous Bounded Domains}, ser. Mathematics and its Applications,  Science
Press, Beijing,  2005


\bibitem{AWeil}A. Weil, {\it Introduction a l'\'etude des vari\'et\'es
    k\"ahl\'eriennes},  Actualit\'es scientifiques et industrielles {\bf
    1971},  Hermann, Paris,  1957 

\bibitem{Y08} J.-H. Yang, \textit{A partial Cayley transform for
Siegel-Jacobi disk}, J. Korean Math. Soc.  {\bf 45}  (2008) 781--794

\bibitem{Y10} J.-H. Yang, \textit{Invariant metrics and Laplacians on the
Siegel-Jacobi disk}, Chin. Ann. Math.  {\bf 31B}   (2010) 85--100


\bibitem{jae} J.-H. Yang,  Y.-H. Yong,  S.-N. Huh,  J.-H.  Shin and 
  H.-G.  Min,
  {\it Sectional curvatures of the Siegel-Jacobi space}, 
Bull. Korean Math. Soc. {\bf 50} (2013)  787--799
 



\end{thebibliography}
\end{document}